\documentclass[journal=gmj]{CUP-JNL-DTM}%

\usepackage{float}

\usepackage{graphicx}
\usepackage{multicol,multirow}
\usepackage{amsmath,amssymb,amsfonts}
\usepackage{mathrsfs}
\usepackage{amsthm}
\usepackage{rotating}
\usepackage{appendix}
\usepackage{ifpdf}
\usepackage[T1]{fontenc}
\usepackage{newtxtext}
\usepackage{newtxmath}
\usepackage{textcomp}
\usepackage{xcolor}
\usepackage{lipsum}
\usepackage[colorlinks,allcolors=blue]{hyperref}
\usepackage{ulem}
\usepackage{soul}

\newtheorem{theorem}{Theorem}[section]
\newtheorem{lemma}[theorem]{Lemma}
\theoremstyle{definition}
\newtheorem{remark}[theorem]{Remark}

\numberwithin{equation}{section}

\newtheorem{definition}[theorem]{Definition}
\newtheorem{proposition}[theorem]{Proposition}

\newtheorem{hypothesis}[theorem]{Hypothesis}
\usepackage{verbatim}

\newcommand{\Tr}{{\operatorname{Tr}}}
\newcommand*{\bigchi}{\mbox{\large$\chi$}}
\DeclareMathOperator\erf{erf}

\newcommand{\sech}{\operatorname{sech}}
\newcommand{\R}{{\mathbb R}}

\newcommand{\C}{{\mathbb C}}

\newcommand{\rexp}{\operatorname{e}} 
\DeclareMathOperator{\diag}{diag}
\renewcommand{\Re}{\operatorname{Re}}
\renewcommand{\Im}{\operatorname{Im}}

\newcommand{\bfA}{{\mathbf{A}}}

\newcommand{\bfI}{{\mathbf{I}}}

\newcommand{\bfQ}{{\mathbf{Q}}}
\newcommand{\bfR}{{\mathbf{R}}}

\newcommand{\bfY}{{\mathbf{Y}}}

\newcommand{\bfp}{{\mathbf{p}}}
\newcommand{\bfq}{{\mathbf{q}}}

\newcommand{\bfu}{{\mathbf{u}}}
\newcommand{\bfv}{{\mathbf{v}}}
\newcommand{\bfw}{{\mathbf{w}}}

\newcommand{\bfy}{{\mathbf{y}}}
\newcommand{\bfz}{{\mathbf{z}}}

\newcommand{\cI}{{\mathcal I}}

\newcommand{\cK}{{\mathcal K}}

\jname{}
\articletype{}

\jyear{}
\jvol{}
\jissue{}

\begin{document}

\begin{Frontmatter}

\title[Stability of Ginzburg-Landau pulses]{Stability of Ginzburg-Landau pulses via Fredholm determinants of Birman-Schwinger operators}

\author[1]{Erika Gallo}
\author[2]{John Zweck}
\author[3]{Yuri Latushkin}

\authormark{E. Gallo, J. Zweck, and Y. Latushkin}

\address[1]{\orgdiv{Department of Mathematics, \orgname{The University of Texas at Arlington}, \orgaddress{\city{Arlington},  \state{TX}, \postcode{76019}, \country{USA}}.
\email{erika.gallo@uta.edu}}}

\address[2]{\orgdiv{Department of Mathematics}, \orgname{New York Institute of Technology}, \orgaddress{\city{New York},  \state{NY}, \postcode{10023}, \country{USA}}. \email{jzweck@nyit.edu}}

\address[3]{\orgdiv{Department of Mathematics, \orgname{University of Missouri}, \orgaddress{\city{Columbia},  \state{MO}, \postcode{65211}, \country{USA}}.
\email{latushkiny@missouri.edu}}}

\keywords{complex Ginzburg-Landau equation, Birman-Schwinger operator, Trace class operator, numerical Fredholm determinant, spectral stability}

\keywords[MSC Codes]{\codes[Primary]{35Q56, 65R20, 47B10}; \codes[Secondary]{47G10, 37L15}}

    
\abstract{We introduce a numerical method to determine the stability of stationary pulse solutions of the  complex Ginzburg-Landau equation.  The method involves the computation of the point spectrum of the first-order linear differential operator with matrix-valued coefficients on the real line obtained by linearizing the Ginzburg-Landau equation about a stationary pulse. Applying a general theory of Gesztesy, Latushkin, and Makarov, we show that this point spectrum is given by the set of zeros of a 2-modified Fredholm determinant of a  Hilbert-Schmidt, Birman-Schwinger operator. We establish conditions which guarantee that this operator is trace class.  Applying  results of Bornemann on the numerical computation of Fredholm determinants, we  obtain a bound on the error between the regular Fredholm determinant of the trace class operator and its numerical approximation by a matrix determinant. We verify the new  numerical Fredholm determinant method for computing the point spectrum of a Ginzburg-Landau pulse by exhibiting excellent agreement  with previous methods. This new approach avoids the challenge of solving the numerically stiff system of equations for the matrix-valued Jost solutions, and it opens the way for the spectral analysis of breather solutions of nonlinear wave equations, for which an Evans function does not exist.}


\end{Frontmatter}

\section{Introduction}

The cubic-quintic complex Ginzburg-Landau equation (CGLE)
is a fundamental model for nonlinear waves and coherent structures
that arise in fields such as  nonlinear optics and condensed matter physics~\cite{akhmediev2008dissipative,RMP74p99}. 
The CGLE supports a wide variety of  solutions, including 
stationary pulses,  breathers,
fronts, exploding solitons, and chaotic solutions~\cite{akhmediev2008dissipative}. 
An important engineering application of the CGLE is
that it provides a qualitative model
for the generation of femtosecond
pulses in fiber lasers~\cite{SIREV48p629}.
In particular, the equation includes large dissipative terms that model the linear filtering  and nonlinear saturable gain and loss that is  a feature of these systems. 

Although  the stability of pulse solutions of the CGLE has been 
theoretically studied~\cite{PhysicaD116p95,PhysicaD124p58,PRA86p033616} 
in the special case that the dissipative terms are small,
the introduction of large dissipative terms gives rise to new classes of solutions. 
While analytical solutions have been found when special relations
hold between the coefficients~\cite{PLA372p3124,JOSAB13p1439,PRE55p4783}, these solutions are unstable in the anomalous dispersion regime~\cite{JOSAB31p2914}.
Moreover, in the case of large dissipative effects 
there are no general results on the stability of numerically determined
pulse solutions that arise in the modeling of fiber lasers.

Many of the theoretical results concerning the stability of  solutions of nonlinear wave equations are based on an analysis of the Evans 
function~\cite{PhysicaD116p95,Indiana53p1095,Kap,JOSAB15p2757,PhysicaD124p58,PTRSLA340p47}. 
Considerable effort has also gone into the development of 
computational  Evans function methods~\cite{PRSL2001p257,PhysicaD172p190,SINUM53p2329,PhysicaD67p45}.
The major challenge with these methods is the need to compute bases for the space of  Jost solutions. The Jost functions are solutions of a 
system of variable-coefficient linear differential equations
with a prescribed set of asymptotic decay rates at spatial infinity. 
Because these decay rates can cover a wide range, the system is 
numerically stiff, which poses significant challenges for numerically 
computing the growth  of Jost solutions from  negative (or positive) 
spatial infinity towards the center of the stationary pulse. Furthermore, the stiffness of the system increases markedly for high-dimensional systems. 
Breather solutions of the CGLE, which are of particular interest in the modeling of modern short-pulse fiber lasers~\cite{shinglot2023essential,shinglot2022continuous,shinglot2024floquet,TURITSYN2012}, are 
pulses that evolve periodically in the temporal variable. By employing
a Fourier series basis in the periodic temporal variable, the linearization of the CGLE about a breather gives rise to an infinite dimensional system. In their paper on the infinite dimensional Evans function for reaction-diffusion systems, Latushkin and Pogan~\cite{latushkin2015infinite}
showed that the analytical Evans function is not even guaranteed to exist 
except when the nonlinearity in the system is small. A practical way
to overcome this shortcoming of the computational Evans function is to 
simply compute the spectrum of a matrix discretization of the linearized
differential operator. For example, matrix discretizations of the monodromy operators
associated with  breather solutions have recently proved to be  effective~\cite{cuevas2017floquet,shinglot2022continuous,shinglot2024floquet}.
However, naive matrix discretizations of linear differential operators
that do not explicitly account for the possible very slow decay of eigenfunctions 
can fail~\cite{barashenkov2000oscillatory,PhysicaD172p190}, especially for eigenvalues that
emerge out of the essential spectrum~\cite{JOSAB15p2757,PhysicaD124p58,shen2016spectra}. 
Indeed the failure of such methods was a primary motivation for the
development of computational Evans function methods. 
Taken together, these observations
highlight the need for an alternative approach to the analysis and numerical computation of the stability of pulse  and breather solutions of nonlinear waves equations.

In 2007, Gesztesy, Latushkin, and Makarov~\cite{EJF} 
established a general theory which showed that, up to  a nonvanishing  factor, the  Evans function, $E=E(\lambda)$, is equal to the 2-modified
Fredholm determinant, $\operatorname{det}_2(\mathcal I+\mathcal K(\lambda))$, of a certain Birman–Schwinger integral operator, 
$\mathcal K=\mathcal K(\lambda)$. Here, $\lambda\in \mathbb C$ denotes the spectral parameter. 
This Birman–Schwinger operator is defined in terms of a matrix-valued,
semi-separable Green’s kernel on the real line. 
In particular, counted with multiplicity, 
the set of zeros of the Evans function coincides with that of the
Fredholm determinant. 
Consequently, the point spectrum of the operator obtained by linearizing the governing nonlinear wave equation about a stationary pulse
is given by the set of zeros of this infinite-dimensional Fredholm determinant. 
Fortuitously, in 2008 Bornemann~\cite{Bornemann} showed that
the Fredholm determinant of a Hilbert-Schmidt operator
on a finite interval can be numerically approximated
by a finite-dimensional matrix determinant constructed
with the aid of a numerical quadrature method. 

The results of Gesztesy, Latushkin, and Makarov~\cite{EJF} and 
Bornemann~\cite{Bornemann} open a window
to an alternate approach for determining the linear stability 
of pulse and breather solutions of nonlinear wave equations.
The purpose of this paper is to investigate the feasibility of this approach in the context of stationary pulse solutions of the CGLE.
Our results build on our recent prior work on a smoothness and decay
criterion on the kernel which guarantees that a Hilbert-Schmidt operator with a 
matrix- or operator-valued kernel on the real line is trace class~\cite{zweck2024regularity}, and on our extension of Bornemann's numerical Fredholm determinant to matrix-valued kernels on the  entire real line~\cite{GZL2025NumericalFredholm}.
Although our results are specific to pulse solutions of the CGLE, we anticipate that the approach can also be applied to breather solutions of the CGLE and to stationary and breather solutions of other nonlinear wave equations.

The results in the paper can be summarized as a follows. 
In section~\ref{sec:Background}, we first recall a  result of 
Zweck, Latushkin, and Gallo~\cite{zweck2024regularity} which states that
a Hilbert-Schmidt operator with a matrix-valued kernel that is  exponentially decaying and Lipschitz-continuous
 is trace class. Then we review results of Bornemann~\cite{Bornemann}, 
and Gallo, Zweck and Latushkin~\cite{GZL2025NumericalFredholm}
on the numerical computation of Fredholm determinants of trace class and Hilbert-Schmidt
operators with matrix-valued kernels. In section~\ref{linearizationsection}, we define the 
second-order, variable-coefficient differential operator, $\mathcal L$, obtained by linearizing the complex Ginzburg-Landau equation about a stationary pulse solution, $\psi$, and the associated constant-coefficient asymptotic
operator, $\mathcal L_\infty$. 
The linear stability of $\psi$ is determined by the spectrum of $\mathcal L$,
which is the union of the essential and point spectra.
Next, we convert the eigenproblem for $\mathcal L_\infty$
to an unperturbed  first-order system of ordinary differential equations
whose solvability characterizes the essential spectrum. 
As in \cite{Kap,EssSpec} this construction enables us to derive an analytical formula for the
essential spectrum.
Similarly, we convert the eigenproblem for $\mathcal L$
to a perturbed  first-order system of ordinary differential equations
whose solvability characterizes the point spectrum.

The results in sections~\ref{diagonalizationsection} and \ref{BSsection}
build on a general theory of Gesztesy, Latushkin, and Makarov~\cite{EJF} which studies connections between unperturbed and perturbed matrix-valued systems of linear
ordinary differential equations on the real line. In particular, they
provide an intrinsic definition of the Evans
function in terms of  generalized matrix-valued Jost solutions
and prove that the  Evans function, which is a matrix determinant, 
coincides with the 2-modified Fredholm determinant of a
Birman–Schwinger-type integral operator, up to an explicitly computable nonvanishing factor.
Here, we work  out the details of this general theory  in the case that the linear differential operator is given by the linearization, $\mathcal L$, of the CGLE. 
In section~\ref{diagonalizationsection}, we diagonalize the constant-coefficient
unperturbed system.
In section~\ref{BSsection}, we review the definition of the Birman-Schwinger integral operator, $\mathcal K(\lambda)$, associated with the variable-coefficient perturbed system.
In particular, we prove that
$\lambda$ is an eigenvalue of $\mathcal L$ if and only if $\mathcal I + \mathcal K(\lambda)$
is invertible. Unlike the approach taken by Gesztesy, Latushkin, and Makarov~\cite{EJF}, we obtain this result directly rather than via the Evans function of $\mathcal L$.

In section~\ref{HSKernelSection}, we show that if $\psi$ is bounded and 
decays exponentially, then the Birman-Schwinger operator is Hilbert-Schmidt and the 
point spectrum of $\mathcal L$ is given by the zero set of the 2-modified Fredholm determinant, $\det_{2}(\mathcal{I} + \mathcal{K}(\lambda))$. 
However, using the diagonalization result in section~\ref{diagonalizationsection}, we 
derive a $\lambda$-dependent estimate 
which strongly suggests that $\det_{2}(\mathcal{I} + \mathcal{K}(\lambda))$
blows up as $\lambda$ approaches the edge of the essential spectrum. This blow-up
phenomenon, which is confirmed by numerical results in section~\ref{sec:NumResults} and further theory in appendix~\ref{App:TraceK},
is problematic for the numerical computation of eigenvalues of $\mathcal L$ that are near
the edge of the essential spectrum~\cite{Kap,shen2016spectra}. 
In section~\ref{Lipschitzsection}, we establish conditions on the parameters in the CGLE and on the pulse, $\psi$, which guarantee that the matrix-valued kernel of $\mathcal{K}(\lambda)$ is globally Lipschitz-continuous on the real line. 
While this result is specific to the CGLE, 
the nature of the proof suggests that a similar result should hold for any $C^1$ stationary
pulse solution of a nonlinear wave equation for which the operator $\mathcal L$ is of the form $\mathcal L = \mathbf B(x)\partial_x^2 + \mathbf M(x)$, where $\mathbf B(x)$ and $\mathbf M(x)$ are matrix-valued functions.

The main theorems of the paper are given in section~\ref{sec:MainTheorems}. 
First, we apply the results in sections~\ref{sec:Background} and \ref{Lipschitzsection} 
to  identify conditions under which $\mathcal{K}(\lambda)$ is trace class.  
In this situation, as a corollary of a general theorem of 
Gesztesy, Latushkin, and Makarov~\cite{EJF},
we conclude that the regular Fredholm determinant, ${\det}_1(\mathcal I + \mathcal K(\lambda))$, is equal to the Evans function, $E(\lambda)$. 
Given the challenges of numerically solving for the Jost functions~\cite{humpherys2006efficient},
this novel result suggests that an alternate method for numerically analyzing the stability
of nonlinear waves is to compute the zeros of the
Fredholm determinant, ${\det}_1(\mathcal I + \mathcal K(\lambda))$.
To further justify this approach, we apply the general results on numerical Fredholm determinants reviewed in section~\ref{sec:Background} to   obtain a bound on the error between the regular Fredholm determinant of the trace class operator, $\mathcal K(\lambda)$, and its numerical approximation by a matrix determinant.

In section~\ref{sec:NumResults}, we will apply the results in section~\ref{sec:MainTheorems} to
numerically compute $\sigma_{\rm pt}(\mathcal L)$ for both the hyperbolic secant solution
of the nonlinear Schr\"odinger equation (NLSE) and for 
a numerically computed stationary solution of the CGLE.
To validate these results,   
in appendix~\ref{App:EvansSech}, for the hyperbolic secant solution, we derive an 
analytical formula for the Evans function 
defined in \cite{EJF},  
and in appendix~\ref{App:TraceK}, for an 
arbitrary CGLE pulse, 
we derive a formula which shows how 
 the trace of $\mathcal K(\lambda)$ depends on $\lambda$.
To further validate the results for the CGLE pulse, we compare the 
point spectrum to that obtained by Shen, Zweck, and Menyuk~\cite{shen2016spectra} 
using a method that is similar in spirit to numerical Evans function methods, but which involves the iterative solution of a nonlinear eigenproblem. In both cases, the  agreement between the two methods is excellent.

\section{Background on the numerical approximations of Fredholm determinants}
\label{sec:Background}

We refer the reader to Teschl~\cite{TeschlFA}, Simon~\cite{Simon},
Bornemann~\cite{Bornemann}, 
Gohberg, Goldberg, and Krupnik~\cite{GGK},
Gallo, Zweck and Latushkin~\cite{GZL2025NumericalFredholm}, and 
Zweck, Latushkin and Gallo~\cite{zweck2024regularity}
for  background material 
on the spaces of  trace class and Hilbert-Schmidt operators and  their Fredholm determinants. 
We let $\mathcal B_1(\textsf H)$ denote the space of trace class and 
$\mathcal B_2(\textsf H)$ denote the space of Hilbert-Schmidt operators  on a separable
Hilbert space, $\textsf H$. 
If $\mathcal K \in \mathcal B_1(\textsf H)$,  
we let $\det_1(\mathcal I + \mathcal K)$ denote the regular Fredholm determinant of $\mathcal K$, and if $\mathcal K \in \mathcal B_2(\textsf H)$,
we let $\det_2(\mathcal I + \mathcal K)$ denote the 2-modified Fredholm determinant of $\mathcal K$. 

\begin{theorem}[{\cite[Theorem 6.1]{zweck2024regularity}}]\label{TraceClassThmRealLine}
Let $\mathbf K\in L^2(\mathbb R\times \mathbb R, \mathbb C^{k\times k})$ be a 
Lipschitz continuous,  matrix-valued kernel such that 
there is an $R>0$ so that for all $|x|$, $|y|>R$
\begin{equation}\label{KandDerivsexp}
     \operatorname{max} \{  \|\mathbf{K}(x,y)\|, \|\partial_x \mathbf{K}(x,y)\|, \|\partial_y \mathbf{K}(x,y)\| \}
     \leq C e^{-\alpha|x-y|},
    \end{equation}
    for some $C, \alpha > 0$, where $\|\cdot\|$ is a matrix norm.
Let $\mathcal K \in \mathcal{B}_2(L^2(\mathbb R,\mathbb{C}^k))$
be the Hilbert-Schmidt operator with kernel $\mathbf K$. 
Then $\mathcal{K} \in \mathcal{B}_1(L^2(\mathbb R, \mathbb{C}^k))$ is trace class.
\end{theorem}

Let $\mathcal K \in \mathcal B_p(L^2(\mathbb R, \mathbb C^k))$ be a trace class ($p=1$)
or Hilbert-Schmidt ($p=2$) operator with a matrix-valued kernel $\mathbf{K}\in L^2(
\mathbb R\times \mathbb R, \mathbb C^{k\times k})$. 
Building on work of Bornemann~\cite{Bornemann}, in Gallo, Zweck and Latushkin~\cite{GZL2025NumericalFredholm} we showed how to numerically
approximate ${\det}_p(\mathcal I +  \mathcal K)$ by the determinant of a 
block matrix defined in terms of the kernel $\mathbf K$. 
The idea is to first truncate $\mathcal K$ to an operator, $ \mathcal{K}|_{[-L,L]}$, 
 on a finite interval $[-L,L]$,
and then use a quadrature rule to approximate the Fredholm determinant of 
$ \mathcal{K}|_{[-L,L]}$  by the determinant of a block matrix. 
In \cite{GZL2025NumericalFredholm} the authors obtained the following error estimates for  these approximations.

 Let $ \mathcal{K}|_{[-L,L]} := P_L \circ \mathcal{K} \circ \iota_{L}$,
    where $\iota_L : L^2([-L,L],\mathbb{C}^k) \rightarrow L^2(\mathbb{R},\mathbb{C}^k)$ is the inclusion operator and $P_L : L^2(\mathbb{R},\mathbb{C}^k) \rightarrow L^2([-L,L],\mathbb{C}^k)$ is the  projection operator given by
      $  (P_L \boldsymbol{\psi})(x) = \bigchi_{[-L,L]}(x) \boldsymbol{\psi}(x)$,
    where $\bigchi_{[-L,L]}$ is the characteristic function of $[-L,L]$. Since 
    the operators, $\iota_L$ and $P_L$, are bounded and 
    $\mathcal B_p$ is an ideal, $\mathcal{K}|_{[-L,L]} \in \mathcal B_p(L^2([-L, L],\mathbb{C}^k)).$ 
 
\begin{theorem}[{\cite[Theorem~5.1]{GZL2025NumericalFredholm}}]
\label{BackgroundTruncationTheorem}
\label{truncthm}
Let $\mathcal{K} \in \mathcal B_p(L^2(\mathbb{R},\mathbb{C}^k))$, for $p=1$ or $2$, and suppose
that $\exists \, C, a > 0$ such that
\begin{equation}\label{truncexpass}
    |{K}_{ij}(x,y)| \leq Ce^{-a(|x| + |y|)}, \,\, \forall i,j \in \{1,...,k\}, \forall x, y, \in \mathbb{R}.
\end{equation}
Then 
\begin{equation}\label{eq:truncthm}
    \left| {\det}_p(\mathcal{I} +  \mathcal{K}) - {\det}_p (\mathcal{I} +  \mathcal{K}|_{[-L,L]})\right| \leq e^{-aL} \boldsymbol{\Phi} \left( \frac{2Ck}{a} \right),
\end{equation}
where 
\begin{equation}\label{phisum}
    \boldsymbol{\Phi}(z) = \sum_{n=1}^{\infty} \frac{n^{(n+2)/2}}{n!} z^n.
\end{equation}
\end{theorem}

\begin{remark}
Bornemann \cite{Bornemann} shows that
$\boldsymbol{\Phi}(z) \leq z \boldsymbol{\Psi}(z\sqrt{2} e)$,
where
\begin{equation}
    \boldsymbol{\Psi}(z) = 1 + \frac{\sqrt{\pi}}{2}z e^{z^2/4} \left[1 + \erf\left(\frac{z}{2}\right) \right].
\end{equation}
\end{remark}

The matrix approximation of the Fredholm determinant of $\mathcal K|_{[-L,L]}$ is defined
in terms of a quadrature rule  for scalar-valued functions, $f :[-L,L]\to\mathbb{C}$.
Specifically, we consider quadrature rules  of the form~\cite{atkinson2008introduction}
\begin{equation}
    Q_M(f) = \sum_{i=1}^M w_i f(x_i),
\end{equation}
that are defined in terms of $M$ nodes, $-L \leq x_1 < x_2 < \dots < x_M \leq L$, and positive weights $w_1, \dots, w_M.$ We suppose that this family of quadrature rules
converges for continuous functions in that 
\begin{equation}
    Q_M(f) \to \int_{-L}^L f(x)dx, \,\, \text{ as } M \to \infty,
\end{equation}
for all $f \in C^0([-L,L],\mathbb{C})$.
For the results in this paper we employ the  composite Simpson's quadrature rule
given by
\begin{equation}\label{Qab}
    Q_{M}(f) = \sum_{k=1}^{2M+1} w_k f(x_k),
\end{equation}
where 
$w_1 = \frac{\Delta x_1}{6}$,  $w_{2M+1} = \frac{\Delta x_M}{6}$, 
$w_{2j} = \frac{2 \Delta x_j}{3}$, and $w_{2j+1} = \frac{(\Delta x_j + \Delta x_{j+1})}{6}$ for $j=1, \dots, M-1$.

Let $\mathbf{K}_Q \in \mathbb{C}^{kM \times kM}$  be the $M \times M$ block matrix 
\begin{equation}\label{QM4}
    \mathbf{K}_Q = \begin{bmatrix} w_1\mathbf{K}(1,1) & w_2 \mathbf{K}(1,2) & \dots & w_M \mathbf{K}(1,M)\\
    w_1 \mathbf{K}(2,1) & w_2 \mathbf{K}(2,2) & \dots & w_M \mathbf{K}(2,M)\\
    \vdots & & & \vdots\\
    w_1 \mathbf{K}(M,1) & w_2 \mathbf{K}(M,2) & \dots & w_M \mathbf{K}(M,M)
    \end{bmatrix},
\end{equation}
where 
   $ \mathbf{K}(\alpha, \beta) := \mathbf {K}(x_{\alpha}, x_{\beta}) \in \mathbb{C}^{k \times k}$
is the $k \times k$ matrix obtained by evaluating the kernel, $\mathbf K$, at the nodes  $x_{\alpha}, x_{\beta} \in \{x_i\}_{i=1}^M$ of the quadrature rule, $Q_M$.
Then the matrix determinant approximations of the Fredholm determinants,
${\det}_p (\mathcal{I} + \mathcal{K}|_{[-L,L]})$,
 are defined by
\begin{equation}\label{eq:dpQdef}
    d_{1,Q} \,\,=\,\, \det[\mathbf{I}_{kM \times kM} +  \mathbf{K}_{Q}]
    \quad\text{and}\quad
    d_{2,Q} \,\,=\,\, e^{-\operatorname{Tr}(\mathbf{K}_Q)}
 {\det}[\mathbf{I}_{kM \times kM} + \mathbf{K}_Q].
    \end{equation}

Let  $C^{r,1}([-L,L], \mathbb{C})$ be the space of functions whose  
$r$-th derivative is Lipschitz continuous.  

\begin{theorem}[{\cite[Theorem~6.3]{GZL2025NumericalFredholm}}]
\label{quadphibound}
Let $\mathcal{K}\in \mathcal B_p(L^2([-L,L],\mathbb{C}^k))$ be an operator with matrix-valued kernel ${\mathbf{K}} \in L^2([-L,L]\times[-L,L], \mathbb{C}^{k \times k}).$ Suppose that for some $1\leq r\leq 4$,
${K}_{ij} \in C^{r-1,1}([-L,L] \times [-L,L], \mathbb{C})$ for all $i,j \in \{1, \dots, k\}$.
Let $Q$ be an adaptive composite Simpson's quadrature rule on $[-L,L]$ with maximum grid 
spacing $\Delta x_{\operatorname{max}}$. Then
\begin{equation}\label{quaderrest}
    |d_{p,Q} - \operatorname{det}_p(\mathcal I + \mathcal K)| \leq \frac{2 (\pi e /8)^r}{(2\pi r)^{1/2}}  \boldsymbol{\Phi}\left(4kL\|\mathbf{K}\|_{r} \right) (\Delta x_{\operatorname{max}})^r,
\end{equation}
where   $ \boldsymbol{\Phi}$ is defined in \eqref{phisum}, and
\begin{equation}\label{newnorm}
\|\mathbf{K}\|_{r} = \max\limits_{i+j\leq r} \| \partial^i_x\partial^j_y \mathbf{K} \|_{L^\infty([-L,L]^2,\mathbb C^{k\times k})}.\end{equation}
\end{theorem}

\begin{remark}
Since $\Phi$ grows super-exponentially as a function of it argument the constants
in the bounds in \eqref{eq:truncthm} and \eqref{quaderrest} are likely to be far from optimal.
This is because these bounds  ignore all signs in the determinant.
\end{remark}

\section{Linearization of the Complex Ginzburg-Landau Equation}\label{linearizationsection}

We consider stationary solutions of the cubic-quintic complex Ginzburg-Landau equation (CGLE),
\begin{equation}\label{cqcgle}
    i\psi_t + \frac{D}{2}\psi_{xx} + \gamma |\psi|^2\psi + \nu |\psi|^4\psi = i\delta \psi + i \epsilon |\psi|^2 \psi + i\beta \psi_{xx} + i \mu |\psi|^4 \psi,
\end{equation}
that are of the form
\begin{equation}
    \psi(t,x) = e^{-i \alpha t}\,\psi(x), \quad \text{for } t \geq 0, x\in \mathbb R,
\end{equation}
and for some phase change $\alpha$. 
We assume throughout that   $(D, \beta) \neq (0,0)$.
 As in \cite{EssSpec},  we set $\boldsymbol{\psi} = [\text{Re}({\psi}) \,\,\, \text{Im}({\psi})]^T: \mathbb R \to \mathbb R^2$
and reformulate \eqref{cqcgle}  to obtain
\begin{equation}\label{cq3}
    \partial_t \boldsymbol{\psi} = \left( \textbf{B}\partial_x^2 + \textbf{N}_0 + \textbf{N}_1|\boldsymbol{\psi}|^2 + \textbf{N}_2|\boldsymbol{\psi}|^4\right) \boldsymbol{\psi},
\end{equation}
where 
\begin{eqnarray}\label{eq:CGLE2x2matrices}
    \textbf{B} = \begin{bmatrix}
    \beta & -\frac{D}{2}\\
    \frac{D}{2} & \beta 
    \end{bmatrix}, \,\,\,\,
\textbf{N}_0 = \begin{bmatrix}
\delta & -\alpha\\ \alpha & \delta
\end{bmatrix}, \,\,\,\,
\textbf{N}_1 = \begin{bmatrix}
\epsilon & -\gamma\\ \gamma & \epsilon
\end{bmatrix},\,\,\,\,
\textbf{N}_2 = \begin{bmatrix}
\mu & -\nu \\ \nu & \mu
\end{bmatrix}.
\end{eqnarray}
Linearizing (\ref{cq3}) about a stationary solution $\boldsymbol{\psi},$ we obtain the equation 
\begin{equation}\label{linbasic}
    \partial_t \mathbf{p} = \mathcal{L}\mathbf{p},\quad \text{with }
    \mathcal{L}= \textbf{B}\partial_x^2 + \widetilde{\textbf{M}}(x),
\end{equation}
where $\widetilde{\textbf{M}}= \widetilde{\textbf{M}}(x)$ is the multiplication operator given by
\begin{equation}\label{MTilde}
    \widetilde{\textbf{M}} = \textbf{N}_0 + \textbf{N}_1|\boldsymbol\psi|^2 + \textbf{N}_2|\boldsymbol\psi|^4 + \left( 2\textbf{N}_1 + 4\textbf{N}_2|\boldsymbol\psi|^2\right)\boldsymbol\psi \boldsymbol\psi^T.
\end{equation}
In \eqref{linbasic} we regard $\mathbf p$ as a mapping, $\mathbf p : \mathbb R \to \mathbb C^2$, since the spectrum of the
non-self adjoint operator, $\mathcal L$, is not constrained to be real.

Zweck et. al \cite{EssSpec} showed that if $\boldsymbol{\psi}$ and its weak derivative, 
$\boldsymbol{\psi}_x$, are bounded on $\mathbb{R}$ and $\boldsymbol{\psi}$ decays exponentially as $x \rightarrow \pm \infty,$ then the linear operator,
$    \mathcal{L}  : H^2(\mathbb{R},\mathbb{C}^2) \subset L^2(\mathbb{R},\mathbb{C}^2) \rightarrow L^2(\mathbb{R},\mathbb{C}^2)$,
is closed, and therefore has a spectrum. The linear stability of the stationary pulse $\boldsymbol{\psi}$ is determined by the spectrum of $\mathcal{L}.$

We recall  \cite{EdmundsEvans,Kato} that 
the \emph{spectrum} of $\mathcal{L}$ is defined by
$    \sigma(\mathcal{L}) := \mathbb{C} \setminus \rho(\mathcal{L})$
where  
$    \rho(\mathcal{L}) := \{ \lambda \in \mathbb{C} | (\mathcal{L} - \lambda \mathcal{I})  \text{ is invertible and } (\mathcal{L} - \lambda \mathcal{I})^{-1} \text{ is bounded}\}$
is the \emph{resolvent set}. 
The \emph{point spectrum } of $\mathcal{L}$ is 
\begin{equation}
    \sigma_{\text{pt}}(\mathcal{L}) := \{ \lambda \in \mathbb{C} \, | \, 
 \text{Ker}(\mathcal{L} - \lambda \mathcal{I}) \neq \{0\} \},
\end{equation}
and the \emph{Fredholm point spectrum} of $\mathcal{L}$ is the subset of $\sigma_{\text{pt}}(\mathcal{L})$ such that
\begin{equation}
    \sigma_{\text{pt}}^{\mathcal{F}}(\mathcal{L}) := \{ \lambda \in \mathbb{C} | (\mathcal{L} - \lambda \mathcal{I}) \text{ is Fredholm, } \text{Ind}(\mathcal{L} - \lambda \mathcal{I}) = 0, \text{ and Ker}(\mathcal{L} - \lambda \mathcal{I}) \neq 0\},
\end{equation}
where $\text{Ind}$ denotes the Fredholm index. 
Then the \emph{essential spectrum} of $\mathcal{L}$ is defined by 
\begin{equation}
    \sigma_{\text{ess}}(\mathcal{L}) := \sigma(\mathcal{L}) \setminus \sigma_{\text{pt}}^{\mathcal{F}}(\mathcal{L}) .
\end{equation}
Then the spectrum of $\mathcal{L}$ is  given by
    $\sigma(\mathcal{L}) = \sigma_{\text{ess}}(\mathcal{L})\cup \sigma_{\text{pt}}(\mathcal{L})$,
although this union may not be disjoint. 
Both the essential spectrum and the point spectrum of the operator $\mathcal{L}$ in (\ref{linbasic}) are computed with the aid of the asymptotic differential operator, $\mathcal{L}_{\infty}$. To define this operator,
we assume that 
\begin{equation}
    \lim_{x \to \pm \infty} \|\boldsymbol\psi(x)\|_{\mathbb{C}^2} = 0,
\end{equation}
so that
\begin{equation}
   \textbf{ M}_\infty := \lim_{x \to \pm \infty} \widetilde{\textbf{M}}(x) = \textbf{N}_0.
\end{equation}

As in \cite[Definition 3.1]{EssSpec}, the asymptotic differential operator, $\mathcal{L}_\infty$, associated with $\mathcal{L}$ is defined by
\begin{equation}\label{operatorLinf}
    \mathcal{L}_\infty = \textbf{B}\partial_x^2 + \textbf{M}_\infty = \textbf{B}\partial_x^2 + \textbf{N}_0.
\end{equation}
To obtain the spectrum of  $\mathcal{L}_\infty,$ we convert the second-order differential equation $(\mathcal{L}_\infty - \lambda)\textbf{p} = 0$ to the \emph{unperturbed} first-order system,
\begin{eqnarray}\label{Ainfdiffeq}
\partial_x \textbf{u} = \textbf{A}_\infty(\lambda) \textbf{u}, \qquad \text{for }
\textbf{u}  : \mathbb R \to \mathbb C^4,
\end{eqnarray}
by setting $\textbf{u} = [\textbf{p} \,\,\, \textbf{p}_x]^T$,  where
 \begin{equation}\label{A}
     \textbf{A}_\infty(\lambda) = \begin{bmatrix}
     0 & \textbf{I}\\ \textbf{B}^{-1}(\lambda - \textbf{N}_0) & 0
     \end{bmatrix}.
 \end{equation}
 We observe that there is a solution
 $\mathbf p\in H^2(\mathbb R, \mathbb C^2)$ of
 $(\mathcal{L}_\infty - \lambda)\textbf{p} = 0$ if and only if 
there is a solution
 $\mathbf u \in H^1(\mathbb R, \mathbb C^4)$ of \eqref{Ainfdiffeq}.

 It has been shown that
the operator $\mathcal{L}$ is  a \emph{relatively compact perturbation} of $\mathcal{L}_{\infty}$ \cite{Kap,EssSpec}, by which we mean that $\exists \lambda \in \rho(\mathcal{L}_{\infty})$ such that $(\mathcal{L} - \mathcal{L}_{\infty})(\mathcal{L}_{\infty} - \lambda \mathcal{I})^{-1} : \mathbb{R} \rightarrow \mathbb{R}$ is a compact operator. Then by Weyl's essential spectrum theorem \cite{Kap}, 
\begin{equation}\label{sigmas}
\sigma_{\text{ess}} (\mathcal{L}) = \sigma_{\text{ess}}(\mathcal{L}_{\infty}) = \sigma(\mathcal{L}_{\infty}) = \{ \lambda \in \mathbb{C} \, : \, \exists \, \mu \in \mathbb{R} : \det[\textbf{A}_{\infty}(\lambda) - i \mu] = 0 \}. 
\end{equation}
That is, $\lambda \in \sigma_{\text{ess}}(\mathcal{L}_{\infty}) $ if and only if the matrix $\textbf{A}_{\infty}(\lambda)$ has a purely imaginary eigenvalue.
A calculation~\cite{EssSpec} shows that
\begin{equation}\label{ess2}
    \sigma_{\text{ess}}(\mathcal{L}_{\infty}) = \left\{ \lambda \in \mathbb{C} \, : \, \lambda = (\delta \pm i \alpha) - \mu^2 \left( \beta \pm i \frac{D}{2} \right) \, \text{ for some } \mu \in \mathbb{R} \right\},
\end{equation} 
is a pair of lines and is stable when the loss parameter, $\delta$, is negative.

Similarly, to characterize the point spectrum, $ \sigma_{\text{pt}}(\mathcal{L})$, 
we convert the second-order differential equation $(\mathcal{L} - \lambda)\textbf{p} = 0$ to the \emph{perturbed} first-order system
\begin{equation}\label{systempert}
    \partial_x \textbf{u} = [\textbf{A}_{\infty}(\lambda) + \textbf{R}(x)]\textbf{u},
\end{equation}
where 
\begin{equation}\label{Rpertdef}
    \textbf{R}(x) = \begin{bmatrix}
    \textbf{0} & \textbf{0}\\ -\textbf{B}^{-1}\textbf{M}(x) & \textbf{0}
    \end{bmatrix},
\qquad\text{with } 
    \textbf{M}(x) := \widetilde{\textbf{M}}(x) - \textbf{N}_0.
\end{equation}
We observe that $\lambda \in \sigma_{\text{pt}}(\mathcal{L})$ if and only if 
there is a solution 
 $\mathbf u \in  H^1(\mathbb R, \mathbb C^4)$
of \eqref{systempert}.

Since the multiplication operator $\mathbf R(x) \to \mathbf 0$
as $x\to\pm\infty$, solutions of the perturbed problem \eqref{systempert}
converge to solutions of the unperturbed problem \eqref{Ainfdiffeq} as $x\to\pm\infty$.
Recall that the decay rates of solutions to the unperturbed problem depend on the
eigenvalues of the matrix $\textbf{A}_{\infty}(\lambda)$.
Consequently, knowledge of the eigenvalues of the matrix $\textbf{A}_{\infty}(\lambda)$
is helpful for determining the point spectrum of $\mathcal L$.

\section{Diagonalization of the Unperturbed System}\label{diagonalizationsection}

In this section, we calculate the eigenvalues of the matrix $\textbf{A}_{\infty}(\lambda)$ in \eqref{A} as functions of the spectral parameter, $\lambda \in \mathbb{C}$. 
In addition, we identify conditions which guarantee that $\textbf{A}_{\infty}(\lambda)$ is diagonalizable. The explicit diagonalization of $\textbf{A}_{\infty}(\lambda)$
given in Proposition~\ref{prop:diagzbl} below will be used to numerically compute the matrix
exponential, $\exp(\textbf{A}_{\infty}(\lambda)x)$, which is a factor in the kernel
of the Birman-Schwinger operator, $\mathcal K$, given in Proposition~\ref{kktildej2}. 
This calculation is key since the point spectrum   of $\mathcal L$ is given by the zeros  
of the Fredholm determinant of $\mathcal K$. 
In addition, the diagonalization of $\textbf{A}_{\infty}(\lambda)$ 
will be used in Theorem~\ref{Knormbd} to derive a $\lambda$-dependent
bound on the 2-modified Fredholm determinant of $\mathcal K$.

Since we are assuming that  $(D,\beta) \neq (0,0)$  the matrix
\begin{equation}
    \widehat{\textbf{B}} = \begin{bmatrix}
    0 & \textbf{B}\\
    \textbf{B} & 0
    \end{bmatrix}
\end{equation}
is invertible. Premultiplying $\textbf{A}_{\infty} - \sigma \textbf{I}$ by $\widehat{\textbf{B}}$ and applying the Schur determinant formula, we find that 
\begin{equation}\label{eq:BhatAinfeig}
    \det(\widehat{\textbf{B}}(\textbf{A}_{\infty}(\lambda) - \sigma \textbf{I})) = \det(\textbf{B})\det(\lambda - \textbf{N}_0 - \sigma^2\textbf{ B}).
\end{equation}
Therefore, $\sigma$ is an eigenvalue of $\textbf{A}_{\infty}(\lambda)$ if and only if
\begin{equation}\label{det0}
   0\,\,=\,\, \det(\lambda - \textbf{N}_0 - \sigma^2 \textbf{B}) \,\,=\,\,
   \det \begin{bmatrix}
    \lambda - \delta - \sigma^2 \beta & \alpha + \sigma^2 \frac{D}{2}\\
    -(\alpha + \sigma^2 \frac{D}{2}) & \lambda - \delta - \sigma^2 \beta
    \end{bmatrix},
\end{equation}
which implies that the eigenvalues of $\textbf{A}_{\infty}(\lambda)$  
are of the form $\{ \pm\sigma_-, \pm\sigma_+ \}$, where 
\begin{equation}\label{eq:sigmapmdef2}
\lambda - \delta - (\sigma\pm)^2\beta \,\,=\,\, \pm i (\alpha + (\sigma\pm)^2 D/2).
\end{equation}

\begin{remark}
By \cite[Proposition 3.2]{zweck2021essential} and $\eqref{det0}$, the essential spectrum is given by
\begin{equation}
\sigma_{\rm ess}(\mathcal L) \,\,=\,\,\{ \lambda \in \mathbb C \,\,: \,\, 
\lambda= \delta \pm i\alpha - (\beta \pm i D/2) \mu^2, \, \mu\in\mathbb R \},
\end{equation}
which is a pair of half-lines with edges at $\lambda= \delta \pm i\alpha$.
In particular,  $\sigma_{\rm ess}(\mathcal L)$ lies in the left-half plane
if $\delta < 0$ and $\beta \geq 0$.
\end{remark}

Suppose that $\lambda \notin \sigma_{\text{ess}}(\mathcal{L}_{\infty})$. 
By  \eqref{sigmas},
$\sigma_\pm \notin i\mathbb R$ and so $\arg(\sigma_\pm^2)\neq -\pi$.
Therefore, we can define
\begin{equation}\label{eq:sigmapmdef}
\sigma_\pm := \sqrt{a\pm ib}, \quad\text{where}\quad
a = \frac{\beta(\lambda - \delta) - \frac{D \alpha}{2}}{\det \textbf{B}} 
\qquad\text{and}\quad
b = \frac{-[\frac{D }{2}(\lambda - \delta) + \alpha\beta]}{ \det \textbf{B}}
\end{equation}
are complex numbers, and
$\sqrt{\cdot}$ denotes the principal branch of the complex square root.
In particular, 
\begin{equation}\label{eq:kappas}
\sigma_\pm = \kappa_\pm + i \eta_\pm, \quad\text{where } \kappa_\pm, \eta_\pm \in \mathbb R 
\text{ and } \kappa_\pm > 0.
\end{equation}
Since the four eigenvalues of $\textbf{A}_\infty(\lambda)$ are given by 
$\{ \pm \sigma_+, \pm \sigma_-\}$, we have the following result.

\begin{remark}\label{rem:sigmaCA}
Since $a\pm ib$ lies on the negative real axis precisely when $\lambda \in \sigma_{\rm ess}(\mathcal L)$, the functions $\lambda \mapsto \sigma_\pm(\lambda)$
are complex analytic on $\mathbb C\setminus \sigma_{\rm ess}(\mathcal L)$.
\end{remark}

\begin{proposition}\label{thm1ev}
Suppose that $(\beta, D) \neq (0,0)$ and $\lambda \notin \sigma_{\text{ess}}(\mathcal{L}_{\infty}).$ Then the matrix $\textbf{A}_\infty(\lambda)$ has two eigenvalues with positive real part, and two eigenvalues with negative real part.
\end{proposition}

\begin{hypothesis}\label{hyp:BDL}
We assume that $\beta \geq 0$, $(D,\beta) \neq (0, 0)$, and $\lambda \notin \sigma_{\text{ess}}(\mathcal{L}_{\infty})$.
\end{hypothesis}

\begin{proposition}\label{thmSigmaNotDistinct}
Suppose Hypothesis \ref{hyp:BDL} holds.
Then we have the following trichotomy. 
\begin{enumerate}
\item[(1)] $\textbf{A}_{\infty}(\lambda)$ has four distinct eigenvalues.
    \item[(2)] $(D,\alpha) = (0,0)$, and if $\lambda \in \mathbb{R}$, then $\lambda > \delta$. In this
     case $\textbf{A}_{\infty}(\lambda)$ has two repeated eigenvalues \begin{equation}\label{sigDzero}
        \sigma(\textbf{A}_{\infty}(\lambda)) = \pm \sqrt{\frac{\lambda - \delta}{\beta}}.
                \end{equation}\\
         \item[(3)] $D \neq 0$ and $\frac{\alpha}{D} < 0$.  In this case 
         $\lambda = \frac{-2\alpha \beta}{D} + \delta$ and
        $\textbf{A}_{\infty}(\lambda)$ has two repeated eigenvalues
        \begin{equation}\label{sigDnonzero}
            \sigma(\textbf{A}_{\infty}(\lambda)) = \pm \sqrt{\frac{-2\alpha}{D}}
            = \pm \sqrt{\frac{\lambda - \delta}{\beta}}.
        \end{equation}
 \end{enumerate}
 \end{proposition}

\begin{proof}
First we note that $0$ is not an eigenvalue of $\textbf{A}_{\infty}(\lambda)$, since
otherwise by \eqref{eq:sigmapmdef2}, $\lambda = \delta \pm i \alpha \in \sigma_{\rm ess}(\mathcal L)$, which contradicts Hypothesis~\ref{hyp:BDL}. Therefore $(a,b)\neq(0,0)$.
The three cases of the trichotomy are 
\begin{enumerate}
\item[(1)] $b\neq 0$: In this case there are clearly four distinct eigenvalues.
\item[(2)] $b=0$ and $D=0$:  In this case $\alpha=0$  by \eqref{eq:sigmapmdef} and Hypothesis~\ref{hyp:BDL}.
Therefore, by \eqref{eq:sigmapmdef},
$    \sigma^2 = \frac{\lambda - \delta}{\beta}$.
If $\lambda \in \mathbb{R}$ and $\lambda \leq \delta,$ then $\sigma^2 \leq 0,$ so $\sigma$ is pure imaginary, which implies that $\lambda \in \sigma_{\text{ess}}(\mathcal{L}_{\infty}),$ contrary to our hypothesis. Therefore, if $\lambda \in \mathbb{R},$ then $\lambda > \delta$ must hold.
\item[(3)] $b=0$ and $D\neq0$: In this case  \eqref{eq:sigmapmdef} implies that
$\lambda - \delta = \frac{-2 \alpha \beta}{D}$ and that 
$\sigma^2 =  a = \frac{-2\alpha}{D}$. 
If $\frac{\alpha}{D} \geq 0$ then $ \sigma^2 \leq 0$ and so $\sigma$ is pure imaginary,
which implies that $\lambda \in \sigma_{\text{ess}}(\mathcal{L}_{\infty})$, contrary to our hypothesis. Therefore, $\frac{\alpha}{D} < 0$ must hold.
\end{enumerate}
\end{proof}


\begin{proposition}\label{prop:diagzbl} 
Suppose that Hypothesis~\ref{hyp:BDL} holds. Then $\textbf{A}_{\infty}(\lambda)= \textbf{P} \textbf{D} \textbf{P}^{-1}$ is diagonalizable. 

(i) If  $\textbf{A}_{\infty}(\lambda)$ has four distinct eigenvalues, $\{ \pm \sigma_+, 
\pm \sigma_- \}$, then
\begin{equation}\label{P}
\textbf{P} = \begin{bmatrix}
        i/\sigma_- & -i/\sigma_+ & i/ \sigma_+ & -i / \sigma_-\\
        -1/\sigma_- & -1/\sigma_+ & 1/ \sigma_+ & 1/ \sigma_-\\  
        -i & i & i & -i\\
        1 & 1 & 1 & 1
    \end{bmatrix}, 
       \textbf{D} = \begin{bmatrix}
        -\sigma_- & 0 & 0 & 0\\
        0 & -\sigma_+ & 0 & 0\\
        0 & 0 & \sigma_+ & 0\\
        0 & 0 & 0 & \sigma_-
    \end{bmatrix},
      \textbf{P}^{-1} =  \frac{1}{4}    \begin{bmatrix}
       -i  \sigma_- & -\sigma_- & i & 1\\
       i  \sigma_+ & -\sigma_+ & -i & 1\\
       -i  \sigma_+ & \sigma_+ & -i & 1\\
       i  \sigma_- & \sigma_- & i & 1\\
  
    \end{bmatrix}.
\end{equation}

(ii) If  $\textbf{A}_{\infty}(\lambda)$ has two distinct eigenvalues,
$\pm\sigma$, then one possible diagonalization is given by
\begin{equation}\label{P2}
    \textbf{P} = 
    \begin{bmatrix}
        0 & 1 & 0 & 1\\
        1 & 0 & 1 & 0\\
        0 & -\sigma & 0 & \sigma\\  
        -\sigma & 0 & \sigma & 0
  \end{bmatrix},
  \quad
 \textbf{D} = \begin{bmatrix}
-\sigma & 0 & 0 & 0\\
0 & -\sigma & 0 & 0\\
0 & 0 & \sigma & 0\\
0 & 0 & 0 & \sigma
\end{bmatrix},
\quad
    \textbf{P}^{-1} = \frac{1}{2} \begin{bmatrix}
        0 & 1 & 0 & -1/\sigma\\
        1 & 0 & -1/\sigma & 0\\
        0 & 1 & 0 & 1/\sigma\\
        1 & 0 & 1/\sigma & 0
    \end{bmatrix}.
\end{equation}
\end{proposition}

\begin{remark}\label{rem:Q}
    In both cases, the spectral projection $\textbf{Q}$ onto the stable subspace of $\textbf{A}_{\infty}(\lambda)$ is of the form $\textbf{Q} = \textbf{P}\widehat{\textbf{Q}}\textbf{P}^{-1},$ where  \begin{equation}\label{qhatpform}
    \widehat{\textbf{Q}} = \begin{bmatrix}
        \textbf{I}_{2 \times 2} & \textbf{0}_{2 \times 2}\\
        \textbf{0}_{2 \times 2} & \textbf{0}_{2 \times 2}
    \end{bmatrix}.
\end{equation}
We use this representation of $\mathbf{Q}$ in the numerical computation of the 
kernel of the Birman-Schwinger operator given in Proposition~\ref{kktildej2} below.
\end{remark}

\begin{proof}

Arguing as in the proof of \eqref{eq:BhatAinfeig}, 
$\begin{bmatrix} \mathbf v & \mathbf w\end{bmatrix}^T$ is an eigenvector
of $\textbf{A}_{\infty}(\lambda)$ with eigenvalue, $\sigma$, if and only if
\begin{equation}\label{eq:Ainfblockev}
\begin{bmatrix}
\lambda - \mathbf N_0 & -\sigma \mathbf B \\ -\sigma \mathbf B & \mathbf B 
\end{bmatrix}
\begin{bmatrix}
 \mathbf v \\ \mathbf w
\end{bmatrix}
\,\,=\,\,
\begin{bmatrix}
 \mathbf 0 \\ \mathbf 0
\end{bmatrix}.
\end{equation}
The second row of \eqref{eq:Ainfblockev} implies that $\mathbf w = \sigma \mathbf v$.
Therefore, by the first row  $(\lambda - \mathbf N_0 - \sigma^2 \mathbf B)\mathbf v
=\mathbf 0$. Setting
$\mathbf v = 
\begin{bmatrix}
 v_1 & v_2
\end{bmatrix}^T$,
we find that $\lambda - \delta -\sigma^2\beta = \frac{-v_2}{v_1} (\alpha + \sigma^2 D/2)$.
So by \eqref{eq:sigmapmdef2}, if $\sigma = \sigma_\pm$ then $\frac{-v_2}{v_1} = \pm i$.
In the case of four distinct eigenvalues, the diagonalization in \eqref{P} now follows. 
If instead $\textbf{A}_{\infty}(\lambda)$ has two distinct eigenvalues, then $b = 0$ and $\textbf{A}_{\infty}(\lambda)$ is of  the form
\begin{equation}
    \textbf{A}_{\infty}(\lambda) = \begin{bmatrix}
        \textbf{0}_{2 \times 2} & \textbf{I}_{2 \times 2}\\
        \sigma^2 \textbf{I}_{2 \times 2} & \textbf{0}_{2 \times 2}
    \end{bmatrix},
\end{equation}
where, by Proposition~\ref{thmSigmaNotDistinct}, $\sigma^2 = \frac{\lambda - \delta}{\beta}$. 
The diagonalization~\eqref{P2} now follows. 
 \end{proof}

\section{The Birman-Schwinger Operator for the Perturbed Problem}\label{BSsection}

Throughout this section, we assume that  $\mathbf{A}_{\infty}(\lambda) $ 
satisfies Hypothesis \ref{hyp:BDL}. 
In particular, $\lambda \notin \sigma_{\text{ess}}(\mathcal{L})$. 
Our goal is to characterize the point spectrum of $\mathcal L$ 
that lies outside the essential spectrum
as the zero set of
the Fredholm determinant of an associated family of Birman-Schwinger integral operators, 
$\mathcal K(\lambda)$, parametrized by the spectral parameter, $\lambda$.
In this section, we review the construction of the operators, $\mathcal K(\lambda)$. 
In section~\ref{HSKernelSection} 
we show that $\mathcal K(\lambda)$ is Hilbert-Schmidt, and in 
sections~\ref{Lipschitzsection} and \ref{sec:MainTheorems} we establish conditions 
on the pulse, $\boldsymbol\psi$, which
guarantee that the operator $\mathcal K(\lambda)$ is trace class. 
To simplify notation, in this section we often let $\mathbf{A} := \mathbf{A}_{\infty}(\lambda)$.

The integral operator, $\mathcal K(\lambda)$, is defined in terms of the first-order
perturbed system of ordinary differential equations given in~\eqref{systempert} using
a general construction of Gesztesy, Latushkin, and Makarov~\cite{EJF}.
Let $\mathcal{L}_A, \mathcal{L}_{A+R} : H^1(\mathbb{R}, \mathbb{C}^4) \rightarrow L^2(\mathbb{R}, \mathbb{C}^4)$ be the $\lambda$-dependent operators associated with the unperturbed and perturbed systems (\ref{Ainfdiffeq}) and (\ref{systempert}), respectively; that is,  
\begin{align}
\label{LA}
\mathcal{L}_{A} \,\,&:=\,\, - \partial_x + \mathbf A_{\infty}(\lambda),
\\ 
\label{LAplusR}
\mathcal{L}_{A+R} \,\,&:=\,\, \mathcal{L}_{A} + \mathbf{R}(x).
\end{align}

The first step in the construction of $\mathcal K(\lambda)$ 
is to derive a formula for the Green's operator, $\mathcal{G}_A$, of the unperturbed first-order system~\eqref{Ainfdiffeq}. 
First, we observe that $\textbf{u} \in H^1(\mathbb{R},\mathbb{C}^4)$ solves the unperturbed problem \eqref{Ainfdiffeq} if and only if $\textbf{u} \in \text{Ker}(\mathcal{L}_A). $ By Hypothesis \ref{hyp:BDL} and \eqref{sigmas}, none of the eigenvalues of $\textbf{A}$ are pure imaginary. Consequently, nonzero solutions, $\textbf{u}(x) = e^{\textbf{A}x}\textbf{u}_0$, of 
\eqref{Ainfdiffeq} must grow as either $x \rightarrow \infty$, or $x \rightarrow -\infty$, or both, and so cannot be in $L^2(\mathbb{R}, \mathbb{C}^4).$ Therefore, Ker$(\mathcal{L}_A) = \{0\}.$

By the theory of exponential dichotomies \cite{Palmer},  $\mathcal{L}_A$ is a Fredholm operator with Fredholm index $0.$ Therefore, since $\text{Ker}(\mathcal{L}_A) = \{0\},$ we also have that  $ \text{Coker}(\mathcal{L}_A) = \{0\}.$ Hence, $\mathcal{L}_A$ is bijective and hence is invertible.  In fact, we have an explicit formula for the Green's operator $\mathcal{G}_A = \mathcal{L}_A^{-1}.$
To solve
\begin{equation}
    \textbf{w} = \mathcal{L}_A \textbf{u} = -\textbf{u}' + \textbf{A}\textbf{u}, \,\, \text{ for } \textbf{u} \in H^1(\mathbb R, \mathbb C^4),
\end{equation}
for $\mathbf u$ in terms of $\mathbf w$, we first observe that
\begin{equation}
    -e^{-\textbf{A}x} \textbf{w}(x) = e^{-\textbf{A}x} \textbf{u}'(x) - \textbf{A}e^{-\textbf{A}x}\textbf{u}(x) = (e^{-\textbf{A}x}\textbf{u})'.
\end{equation}
Clearly,
\begin{equation}
    e^{-\textbf{A}x}\textbf{w}(x) = e^{-\textbf{A}x}\textbf{Q}\textbf{w}(x) + e^{-\textbf{A}x}(\textbf{I} - \textbf{Q})\textbf{w}(x),
\end{equation}
where \textbf{Q} is the projection operator onto the stable subspace of \textbf{A}. Consequently, 
\begin{equation}
    \textbf{u}(x) e^{-\textbf{A}x} = \int_{-\infty}^{x} e^{-\textbf{A}y}\textbf{Q}\textbf{w}(y)dy - \int_{x}^{\infty} e^{-\textbf{A}y}(\textbf{I} - \textbf{Q})\textbf{w}(y)dy.
\end{equation}
That is,
\begin{equation}\label{Gakernel}
    \textbf{u}(x) = (\mathcal{G}_A \textbf{w})(x) := \int_{\mathbb{R}}\textbf{G}_A(x-y) \textbf{w}(y) dy,
\end{equation}
where 
\begin{equation}\label{psikernel}
    \textbf{G}_A(x) = \begin{cases}
 -e^{\textbf{A}x}\textbf{Q}, \,\,\, &x \geq 0,\\
 e^{\textbf{A}x}(\textbf{I} - \textbf{Q}),  \,\,\, &x < 0.
    \end{cases}
\end{equation}
We note that the kernel, $\textbf{G}_A(x,y) := \textbf{G}_A(x-y)$, belongs to
$L^2(\mathbb{R} \times \mathbb{R}, \mathbb{C}^{4 \times 4})$ since $\textbf{G}_A(x)$ decays
 exponentially  as $x \rightarrow \pm \infty.$ Therefore, $\mathcal{G}_A$ is a Hilbert Schmidt operator. 

Turning to the perturbed problem, by \eqref{systempert} we know that
$\lambda \in \sigma_{\text{pt}}(\mathcal{L})$ if and only if $\text{Ker}(\mathcal{L}_{A+R}) \neq \{\textbf{0}\}.$ By the theory of exponential dichotomies \cite{Palmer}, $\mathcal{L}_{A+R}$ is Fredholm of index 0, which implies that $\lambda \in \sigma_{\text{pt}}(\mathcal{L})$ if and only if $\mathcal{L}_{A+R}$ is not invertible.

To characterize the invertibility of 
$\mathcal{L}_{A + R}$, it is helpful to derive the polar decomposition of the  perturbation operator, $\textbf{R}(x)$. 
Recall from \eqref{Rpertdef} that $\mathbf R(x)$ is defined in terms of $\mathbf M(x)$  by,
\begin{equation}\label{Rdef}
    \textbf{R}(x) =  \begin{bmatrix}
    \textbf{0} & \textbf{0}\\ -\mathbf {B}^{-1}\textbf{M}(x) & \textbf{0}
    \end{bmatrix}.
\end{equation}
Omitting the dependence on $x$ for now, we recall that
the polar decomposition  of $\textbf{M}$ is of the form  
$\textbf{M} = \textbf{V}\,|\textbf{M}| $
where $|\textbf{M}| = (\textbf{M}^*\textbf{M})^{1/2}$ is positive semi-definite and $\mathbf{V}$ is unitary~\cite{golub2013matrix}.
To obtain the polar decomposition of $\mathbf R$ from that of $\mathbf M$, we first observe that
\begin{equation}
    |\textbf{R}| 
    \,\,=\,\, 
(\textbf{R}^*\textbf{R})^{1/2}
     \,\,=\,\, 
     (\det \textbf{B})^{-1/2} \begin{bmatrix}
    |\textbf{M}| & \textbf{0}\\ \textbf{0} & \textbf{0}
    \end{bmatrix},
\end{equation}
since, by \eqref{eq:CGLE2x2matrices}, $(\det(\mathbf B))^{-1/2} \mathbf B$ is unitary. Next, if we let $\mathbf U$ be the unitary matrix
\begin{equation}
    \textbf{U} 
    \,\, =\,\,
        \begin{bmatrix}
    \textbf{0} & \textbf{I}\\ -(\det\textbf{B})^{1/2}\textbf{B}^{-1}\textbf{V} & \textbf{0}
    \end{bmatrix},
\end{equation}
then we find that $\mathbf R = \mathbf U \,| \mathbf R |$.
Using this polar decomposition, we can factor $\mathbf R$ as 
\begin{equation}
\textbf{R} =\textbf{R}_{\ell}\textbf{R}_{r},
\end{equation}
where 
\begin{equation}\label{eq:redefRrl} 
    \textbf{R}_{\ell} := (\det \textbf{B})^{-1/4}\textbf{U}|\textbf{R}|^{1/2} = \begin{bmatrix}
    \textbf{0} & \textbf{0}\\ -\textbf{B}^{-1}\textbf{M}|\textbf{M}|^{-1/2} & \textbf{0}
    \end{bmatrix}
       \text{ and }
      \textbf{R}_r := (\det \textbf{B})^{1/4}|\textbf{R}|^{1/2} = \begin{bmatrix}
    |\textbf{M}|^{1/2} & \textbf{0}\\ \textbf{0} & \textbf{0}
    \end{bmatrix}.    
\end{equation}

Applying the Birman-Schwinger principle \cite{EJF}, we observe that 
\begin{equation}\label{eq:BSTrick}
\mathcal{L}_{A + R} \,\,=\,\, \mathcal{L}_A + \textbf{R}
\,\,=\,\, \mathcal{L}_A [\mathcal{I} + \mathcal{L}_A^{-1}\textbf{R}]
\,\,=\,\, \mathcal{L}_A [\mathcal{I} + \mathcal{G}_A \textbf{R}_{\ell} \textbf{R}_r].
\end{equation}
This identity motivates the following definition and theorem.

\begin{definition}\label{def:BSoperators}
   The \emph{unsymmetrized} and \emph{symmetrized} \emph{Birman-Schwinger operators} are  the operators
\begin{equation}\label{Kopdef}
\widetilde{\mathcal{K}}(\lambda) \,\,=\,\, \mathcal{G}_A(\lambda) \textbf{R}_{\ell}\textbf{R}_r
\qquad\text{and}\qquad
\mathcal{K}(\lambda) \,\,=\,\, \textbf{R}_r \mathcal{G}_A(\lambda)\textbf{R}_{\ell}.
\end{equation}
\end{definition}

\begin{theorem}\label{LptKtilde}
Suppose that $\lambda \notin \sigma_{\text{ess}}(\mathcal{L})$. Then
     $\lambda \in \sigma_{\text{pt}}(\mathcal{L})$ if and only if  $\mathcal{I} + {\mathcal{K}}(\lambda)$ is not invertible. 
\end{theorem}

\begin{proof}
First, recall that $\lambda \in \sigma_{\text{pt}}(\mathcal{L})$ if and only $\mathcal{L}_{A+R}$ 
is not invertible.
   Now by \eqref{eq:BSTrick}, we see that $\mathcal{L}_{A+R}$ is invertible precisely when both $\mathcal{L}_A$ and $\mathcal{I} + \mathcal{G}_A \textbf{R}_{\ell} \textbf{R}_r$ are invertible. 
    When $\lambda \notin \sigma_{\text{ess}}(\mathcal{L}),$ $\mathcal{L}_A$ is  invertible. So $\mathcal{L}_{A+R}$ is not invertible if and only if $\mathcal{I} + \mathcal{G}_A \textbf{R}_{\ell}\textbf{R}_r$ is not invertible. Finally, we observe that $\mathcal{K}(\lambda)$ and $\widetilde{\mathcal{K}}(\lambda)$ are simultaneously (non)-invertible,  since $(I+BA)^{-1} = I - B(I+AB)^{-1}A$. 
\end{proof}

\section{The Hilbert-Schmidt Kernel}\label{HSKernelSection}

In this section, we provide a condition on the stationary pulse solution, $\psi$,
of the CGLE~\eqref{cqcgle}
 which guarantees that the Birman-Schwinger operator, $\mathcal{K}(\lambda)$,
 in Definition~\ref{Kopdef} is Hilbert-Schmidt. In this situation,  $\lambda \in \sigma_{\rm pt}(\mathcal L)$ if and only if $\lambda$ is a zero of the  $2$-modified Fredholm determinant of 
 $\mathcal{K}(\lambda)$. 
We also derive an estimate which shows how the Hilbert-Schmidt norm of
 $\mathcal{K}(\lambda)$ depends on the spectral parameter, $\lambda$.  
 This estimate enables us to obtain a $\lambda$-dependent bound on the $2$-modified Fredholm determinant of $\mathcal{K}(\lambda)$, which we will use to 
 interpret the results of numerical computations in section~\ref{sec:NumResults}.

\begin{hypothesis}\label{psiexp}
    Let $\psi = \psi(x)$ be a stationary solution of the CGLE~\eqref{cqcgle}. Assume that $\exists \,\, {C}, {a} > 0$ such that 
    \begin{equation}
        |\psi(x)| \leq {C} e^{-{a}|x|}, \,\, \forall x \in \mathbb{R}.
    \end{equation}
\end{hypothesis}

\begin{proposition}\label{kktildej2}
Suppose that Hypothesis~\ref{psiexp} holds. 
Then the Birman-Schwinger operator, $\mathcal{K}(\lambda)$, is Hilbert Schmidt,  
$\mathcal{K}(\lambda)\in \mathcal{B}_2(L^2(\mathbb{R}, \mathbb{C}^4))$. Furthermore, the matrix-valued kernel of  $\mathcal{K}(\lambda)$ is given by 
\begin{equation}\label{opkernel}
    \textbf{K}(x,y;\lambda) = \begin{cases}
    -\textbf{R}_{r}(x)\textbf{Q}(\lambda)e^{\textbf{A}_{\infty}(\lambda)(x-y)}\textbf{Q}(\lambda)\textbf{R}_{\ell}(y), & x \geq y,\\
    \textbf{R}_r(x)(\textbf{I}-\textbf{Q}(\lambda))e^{\textbf{A}_{\infty}(\lambda)(x-y)}(\textbf{I} - \textbf{Q}(\lambda))\textbf{R}_{\ell}(y), & x < y,
    \end{cases}
\end{equation}
where $\textbf{A}_{\infty}(\lambda)$ is the matrix for the unperturbed system  given in \eqref{A},
$\textbf{Q}(\lambda)$
is the spectral projection onto the stable subspace of $\textbf{A}_{\infty}(\lambda)$ given in Remark~\ref{rem:Q}, and $\mathbf R_l$ and 
$\mathbf R_r$ are given by \eqref{eq:redefRrl}.
\end{proposition}

\begin{remark} 
By Remark~\ref{rem:sigmaCA}, Proposition~\ref{prop:diagzbl} and Remark~\ref{rem:Q},
the kernel, $\mathbf K(x,y;\lambda)$, and hence also $\det_p(\mathcal I + \mathcal K(\lambda))$, is defined and complex analytic
for all $\lambda \in \mathbb C\setminus \sigma_{\rm ess}(\mathcal L)$.
\end{remark}

\begin{proof}
Since $\mathcal{K}(\lambda) = \textbf{R}_r \mathcal{G}_A \textbf{R}_{\ell},$ where $\mathcal{G}_A$ is the integral operator  in (\ref{Gakernel}), the kernel of $\mathcal{K}(\lambda)$ is given by 
\begin{equation}
    \textbf{K}(x,y;\lambda) = \begin{cases}
    -\textbf{R}_{r}(x)e^{\textbf{A}_\infty(\lambda)(x-y)}\textbf{Q}(\lambda)\textbf{R}_{\ell}(y), & x \geq y,\\
    \textbf{R}_r(x)e^{\textbf{A}_\infty(\lambda)(x-y)}(\textbf{I} - \textbf{Q}(\lambda))\textbf{R}_{\ell}(y), & x < y.
    \end{cases}
\end{equation}
Equation~\eqref{opkernel} now follows since the projection operator $\textbf{Q}(\lambda)$ commutes with $e^{\textbf{A}_\infty(\lambda)(x-y)}$.

Gesztesy, Latushkin and Makarov~\cite[Lemma 2.9]{EJF} proved a general result that a Birman-Schwinger operator,
$\mathcal{K}(\lambda)$, of the form  \eqref{Kopdef} is Hilbert-Schmidt  
provided that $\|\textbf{R}\|_{\mathbb{C}^{4 \times 4}} \in L^1(\mathbb{R}) \cap L^2(\mathbb{R})$. 
Here, $\| \mathbf R\|_{\mathbb C^{k\times k}}$ denotes the matrix 2-norm of  a $k\times k$ matrix, $\mathbf R$.
Now by \eqref{Rpertdef}, 
   \begin{equation}
        \|\textbf{R}(x)\|_{\mathbb{C}^{4 \times 4}}
         = \|\textbf{B}^{-1}\textbf{M}(x)\|_{\mathbb{C}^{2 \times 2}} \leq \|\textbf{B}^{-1}\|_{\mathbb{C}^{2 \times 2}} \|\textbf{M}(x)\|_{\mathbb{C}^{2 \times 2}}.
    \end{equation}
By  \eqref{Rpertdef}, \eqref{MTilde} and Hypothesis~\ref{psiexp},  
$\|\textbf{M}(x)\|_{\mathbb{C}^{2 \times 2}}$ is bounded on $\mathbb R$ and 
decays exponentially as $x\to\pm \infty$. Therefore, $\|\textbf{R}\|_{\mathbb{C}^{4 \times 4}} \in L^1(\mathbb{R}) \cap L^2(\mathbb{R})$, as required. 
\end{proof}

The following result follows immediately from Theorem~\ref{LptKtilde},
 Proposition~\ref{kktildej2}, and the
fact that if $\mathcal K$ is a Hilbert-Schmidt operator, then $\mathcal I + \mathcal K$
is invertible if and only if ${\det}_2(\mathcal{I} + \mathcal{K}) \neq 0$~\cite{Simon}.

\begin{theorem}\label{KtildeKthm}
Suppose that Hypotheses~\ref{hyp:BDL} and \ref{psiexp} hold.  Then, 
for $\lambda \in \mathbb C \setminus \sigma_{\operatorname{ess}}(\mathcal L)$,
\begin{equation}\label{lambdaiff}
    \lambda \in \sigma_{\operatorname{pt}}(\mathcal{L}) \iff {\det}_2(\mathcal{I} + \mathcal{K}(\lambda)) = 0.
\end{equation}
\end{theorem}

Next we derive a bound for ${\det}_2(\mathcal{I} + \mathcal{K}(\lambda))$ in terms
of the condition number of the matrix, $\mathbf P(\lambda)$, of eigenvectors of 
$\mathbf A_\infty(\lambda)$. Recall that the condition number of an invertible matrix
$\mathbf P\in \mathbb C^{k\times k}$ 
is given by $\operatorname{cond}(\mathbf P) = \| \mathbf P\|_{\mathbb C^{k\times k}}
 \| \mathbf P^{-1}\|_{\mathbb C^{k\times k}}$.  
We first note that if $\psi$ satisfies Hypothesis~\ref{psiexp} then there are constants,
$C_R$ and $a_R$ so that for all $x\in \mathbb R$, 
\begin{equation}\label{eqRest}
  \|\textbf{R}(x)\|_{\mathbb C^{4 \times 4}}  \leq C_R e^{-a_R|x|}.
\end{equation}

\begin{theorem}\label{Knormbd}
Suppose that  $\lambda \notin \sigma_{\rm ess}(\mathcal L)$ and 
$\mathbf R(x)$ satisfies \eqref{eqRest}. Then there is a constant $\Gamma > 0$
so that
\begin{equation}
 |{\det}_2(\mathcal{I} + \mathcal{K}(\lambda))| \leq 
\exp\left(\frac{32 \Gamma C_R^2\operatorname{cond}^2(\mathbf P(\lambda))}{a_R^2}\right).
\end{equation}
\end{theorem}

\begin{remark}\label{rem:condP}
Let $\| \mathbf \cdot \|_F$ denote the Frobenius matrix norm.
Since $\frac 12 \| \mathbf P\|_F \leq \| \mathbf P \|_{\mathbb C^{4\times 4}} \leq \| \mathbf P\|_F $, we know that
$\frac 14 \operatorname{cond}_F(\mathbf P) \leq 
\operatorname{cond}(\mathbf P) 
\leq \operatorname{cond}_F(\mathbf P) $
where
$\operatorname{cond}(\mathbf P) := \|  \mathbf P\|_F  \|  \mathbf P^{-1}\|_F $.
 By Proposition~\ref{prop:diagzbl}, 
 \begin{equation}
 \operatorname{cond}_F^2(\mathbf P) 
 = \left( \frac {1}{|\sigma_-|^2} + \frac {1}{|\sigma_+|^2} + 2 \right)
 \left( {|\sigma_-|^2} + {|\sigma_+|^2} + 2 \right).
 \end{equation}
 We observe that as $\lambda \to \delta  \pm i \alpha$, which by \eqref{sigmas} are the points
 at the edge of the essential spectrum of $\mathcal L$, $\sigma_\pm \to 0$, and hence
$ \operatorname{cond}(\mathbf P) \to \infty$. Therefore, we expect 
${\det}_2(\mathcal{I} + \mathcal{K}(\lambda))$ to blow up as $\lambda \to \delta  \pm i \alpha$.
 This phenomenon could 
be problematic for numerical computation of 
eigenvalues of $\mathcal L$ that are near the edge of the
essential spectrum~\cite{JOSAB15p2757,shen2016spectra}.
\end{remark}

The proof of Theorem~\ref{Knormbd} relies on the following lemma, which we will
also use in section~\ref{Lipschitzsection} below.
For the rest of this section we suppress dependences on $\lambda$
and we assume that $\lambda \notin \sigma_{\rm ess}(\mathcal L)$.
First, we recall from \eqref{eq:kappas} that $\mathbf A_\infty$ has two eigenvalues
with a positive real part, $\sigma_\pm = \kappa_\pm + i \eta_\pm$, where
$\kappa_\pm > 0$. Let
\begin{equation}\label{eq:defkappa12}
\kappa_1 = \min \{ \kappa_+, \kappa_- \}
\qquad\text{and}\qquad
\kappa_2 = \max \{ \kappa_+, \kappa_- \}.
\end{equation}

\begin{lemma}\label{lemma:normKest}
Under the assumptions of Theorem~\ref{Knormbd},
\begin{equation}
\| \mathbf K(x,y) \|_{\mathbb C^{4\times 4}} \,\,\leq\,\,
\begin{cases}
\sqrt{2}\operatorname{cond}(\mathbf P) C_R e^{-a_R ( |x| + |y| )/2 }
e^{-\kappa_{1}(x-y)}
& \quad \text{for } x \geq y,
\\
\sqrt{2}\operatorname{cond}(\mathbf P) C_R e^{-a_R ( |x| + |y| )/2 }
e^{\kappa_{1}(x-y)}
& \quad \text{for } x \leq y.
\end{cases}
\end{equation}
Consequently, there $\exists C_{R}(\lambda)$,  so that for all $i,j$ and all
$(x,y)\in \mathbb R\times\mathbb R$,
\begin{equation}
|K_{ij}(x,y)| \leq C_{R}(\lambda) e^{-a_R ( |x| + |y| )},
\end{equation}
and
\begin{equation}
|K_{ij}(x,y)| \leq C_{R}(\lambda) e^{-\kappa_1  |x-y| }.
\end{equation}
\end{lemma}

\begin{remark}\label{rem:normKest}
Similar bounds hold for $\partial_x \mathbf K$ and $\partial_y \mathbf K$.
\end{remark}

\begin{proof}
By \eqref{opkernel}, when $x\geq y$,
\begin{align}
\| \mathbf K(x,y) \|_{\mathbb C^{4\times 4}} 
\,\,&\leq\,\,
\| \mathbf R_r(x)\|_{\mathbb{C}^{4 \times 4}}
\| \textbf{Q} e^{\textbf{A}_{\infty} (x-y)} \textbf{Q} \|_{\mathbb{C}^{4 \times 4}}
\| \textbf{R}_{\ell}(y)\|_{\mathbb{C}^{4 \times 4}}
\nonumber
\\
\,\,&\leq\,\,
\| \mathbf R(x)\|^{1/2}_{\mathbb{C}^{4 \times 4}}
\| \textbf{R}(y)\|^{1/2}_{\mathbb{C}^{4 \times 4}}
\| \textbf{Q} e^{\textbf{A}_{\infty} (x-y)} \textbf{Q} \|_{\mathbb{C}^{4 \times 4}},
\end{align}
since
   $ \operatorname{max}\{\|\textbf{R}_r\|_{\mathbb{C}^{4 \times 4}}^2, \|\textbf{R}_\ell\|_{\mathbb{C}^{4 \times 4}}^2 \}  \leq \|\textbf{R}\|_{\mathbb{C}^{4 \times 4}}$.
   Next, since $\mathbf A_\infty = \mathbf P \mathbf D \mathbf P^{-1}$ and
$\mathbf Q = \mathbf P \widehat{\mathbf Q} \mathbf P^{-1}$, 
    \begin{align} \nonumber 
      \|\textbf{Q}e^{\textbf{A}_{\infty}(x-y)}\textbf{Q}\|_{\mathbb{C}^{4 \times 4}}
      &= \| \textbf{P} \widehat{\textbf{Q}} e^{\textbf{D}(x-y)} \widehat{\textbf{Q}} \textbf{P}^{-1}\|_{\mathbb{C}^{4 \times 4}} 
    \leq \text{cond}(\textbf{P)}\, \|\widehat{\textbf{Q}}e^{\textbf{D}(x-y)} \widehat{\textbf{Q}}\|_F\\ \nonumber 
    &= \text{cond}(\textbf{P}) \left( |e^{-\sigma_{+}(x-y)}|^2  + |e^{-\sigma_{-}(x-y)}|^2 \right)^{1/2}\\ \nonumber 
    &=
\text{cond}(\textbf{P}) \left( e^{-2\kappa_{+}(x-y)}  + e^{-2\kappa_{-}(x-y)} \right)^{1/2}\\
&\leq \sqrt{2} \,  \text{cond}(\textbf{P}) e^{-\kappa_{1}(x-y)}, 
    \end{align}
    by~\eqref{eq:defkappa12}. 
   When $x\geq y$, the result now follows from \eqref{eqRest}.
When $x\leq y$,  the result follows by a  similar argument and the estimate
\begin{equation}
    \|(\textbf{I} - \textbf{Q}) e^{\textbf{A}_{\infty}(x-y)}(\textbf{I} - \textbf{Q})\|_{\mathbb{C}^{4 \times 4}} \leq \sqrt{2} \,\operatorname{cond}(\textbf{P}) e^{\kappa_{1}(x-y)}.
    \label{eq:condIminusQ}
\end{equation}

\end{proof}

\begin{proof}[Proof of Theorem \ref{Knormbd}]
By Lemma~\ref{lemma:normKest}, and abbreviating $a_R$ to $a$,
\begin{align}
    \nonumber \|\mathcal{K}\|_{\mathcal{B}_2(L^2(\mathbb{R},\mathbb{C}^{4}))}^2 
    & = \int_{-\infty}^\infty \int_{-\infty}^\infty  \| \mathbf K(x,y) \|_{\mathbb{C}^{4 \times 4}}^2 dy dx
\\    
    &\leq 2 C_R^2 \text{cond}^2(\textbf{P}) \int_{-\infty}^\infty \int_{- \infty}^x e^{-a(|x| + |y|)} e^{-2 \kappa_{1}(x-y)} dy dx\\ \nonumber 
    &+ 
 2 C_R^2 \text{cond}^2(\textbf{P})\int_{-\infty}^\infty \int_{x}^{\infty} e^{-a (|x| + |y|)} e^{2 \kappa_{1}(x-y)} dy dx\\ \nonumber 
 &\leq 2 C_R^2 \text{cond}^2(\textbf{P}) \int_{-\infty}^\infty e^{-\frac{a}{2} |x|} \int_{- \infty}^x e^{-\frac{a}{2}(|x| + |y|) - 2 \kappa_{1}(x-y)} dy dx\\
    &+ 
 2 C_R^2 \text{cond}^2(\textbf{P})\int_{-\infty}^\infty e^{- \frac{a}{2} |x|} \int_{x}^{\infty} e^{-\frac{a}{2} (|x| + |y|) + 2 \kappa_{1}(x-y)} dy dx.
\end{align}
Now, for $y \leq x$,
$
    -\frac{a}{2}(|x|+|y|) \leq  -\frac{a}{2}(x-y),
$
and for $y \geq x$,
$
    -\frac{a}{2}(|x| + |y|
    ) \leq  \frac{a}{2}(x-y)
$.
Therefore, 
\begin{align}
   \nonumber  \|\mathcal{K}\|^2_{\mathcal{B}_2(L^2(\mathbb{R},\mathbb{C}^4))}  &\leq 2 C_R^2 \text{cond}^2(\textbf{P}) \int_{-\infty}^\infty e^{-\frac{a}{2} |x|} \int_{- \infty}^x e^{- \left(\frac{a}{2} + 2 \kappa_{1}\right)(x-y)} dy dx\\ \nonumber 
    &\quad+ 
 2 C_R^2 \text{cond}^2(\textbf{P})\int_{-\infty}^\infty e^{- \frac{a}{2} |x|} \int_{x}^{\infty} e^{\left(\frac{a}{2} + 2 \kappa_{1}\right) (x-y)} dy dx\\ \nonumber 
 &=   \frac{16C_R^2 \text{cond}^2 (\textbf{P})}{a(\frac{a}{2} + 2\kappa_{1})} 
 <
 \frac{32 C_R^2 \text{cond}^2(\textbf{P})}{a^2},
\end{align}
since $ \kappa_{1} > 0$.
The result now follows since by~\cite{Simon} there is a constant $\Gamma$ so that
$|{\det}_2(\mathcal{I} + \mathcal{K}(\lambda))| \leq \exp(\Gamma \|\mathcal{K}\|^2_{\mathcal{B}_2(L^2(\mathbb{R},\mathbb{C}^4))})$.
\end{proof}

\section{Lipschitz Continuity of the Kernel}\label{Lipschitzsection}

In order to apply Theorem \ref{TraceClassThmRealLine} to show that $\mathcal{K}$ is trace class and to apply Theorem \ref{quadphibound} to determine the rate of convergence of the numerical approximation of the Fredholm determinant, $\det_p(\mathcal{I} + \mathcal{K})$, we must show that the kernel $\textbf{K}$ is Lipschitz continuous. Throughout this section, we suppose that the stationary pulse, $\psi$, is $C^1$. In Proposition~\ref{prop:Kcts} we show that $\mathbf K(x,y)$ is continuous across the diagonal. 
Recall that the kernel is defined in terms of the matrix-valued function
$\mathbf M : \mathbb R \to \mathbb R^{2\times 2}$ given in 
 \eqref{MTilde} and \eqref{Rpertdef}.
In Theorem~\ref{Lipschitz}, we show that
if $\det(\mathbf M) \neq 0,$ then  the elements of the matrix kernel $\textbf{K}(x,y)$ are $C^1$ functions away from the diagonal. Therefore, under the additional assumption that  the kernel decays exponentially, we conclude that $\mathbf K$ is Lipschitz continuous on $\mathbb R \times \mathbb R$. Finally, in Proposition~\ref{prop:lowroot}, 
we provide conditions on the
parameters in the CGLE~\eqref{cqcgle} and on the maximum amplitude of the pulse which guarantee that $\det(\mathbf M) \neq 0$, and hence that 
$\mathbf K$ is Lipschitz continuous.

Recall that by Proposition \ref{kktildej2}, 
\begin{equation}\label{KformLip}
    \textbf{K}(x,y) = \begin{cases}
    -\textbf{R}_{r}(x)\textbf{Q}e^{\textbf{A}_{\infty}(x-y)}\textbf{Q}\textbf{R}_{\ell}(y), & x \geq y,  \\
    \textbf{R}_r(x)(\textbf{I}-\textbf{Q})e^{\textbf{A}_{\infty}(x-y)}(\textbf{I} - \textbf{Q})\textbf{R}_{\ell}(y), & x < y.
    \end{cases}
\end{equation}
We begin by observing the following remarkable fact.

\begin{proposition}\label{prop:Kcts}
$\mathbf{K}$ is continuous across the diagonal.
\end{proposition}

\begin{proof}
First, we observe that
\begin{equation}  
    \lim_{x - y \to 0^+} \textbf{K}(x,y) 
    = \lim_{x - y \to 0^+} -\textbf{R}_r(x) \textbf{Q} e^{\textbf{A}_{\infty}(x-y)} \textbf{QR}_{\ell}(y)
    = -\textbf{R}_r(x) \textbf{Q} \textbf{R}_{\ell}(x),
\end{equation}
as $\textbf{Q}^2 = \textbf{Q}$.  Second,
\begin{equation} 
    \lim_{x-y \to 0^-} \textbf{K}(x,y) 
    = \textbf{R}_r(x) (\textbf{I}-\textbf{Q}) \textbf{R}_{\ell}(x)
    = -\textbf{R}_r(x) \textbf{Q} \textbf{R}_{\ell}(x),
\end{equation}
as required, since $(\textbf{I}-\textbf{Q})^2 = (\textbf{I}-\textbf{Q})$ and, by \eqref{eq:redefRrl},
\begin{equation}\label{eq:RrRlid}
    \textbf{R}_r(x)\textbf{R}_{\ell}(x) = \begin{bmatrix}
    |\textbf{M}|^{1/2} & \textbf{0}\\ \textbf{0} & \textbf{0}
    \end{bmatrix} \begin{bmatrix}
    \textbf{0} & \textbf{0}\\ -\textbf{B}^{-1}\textbf{M}|\textbf{M}|^{-1/2} & \textbf{0}
    \end{bmatrix} = \begin{bmatrix}
    \textbf{0} & \textbf{0}\\ \textbf{0} & \textbf{0}
    \end{bmatrix}.
\end{equation}
\end{proof}

\begin{remark}
The identity~\eqref{eq:RrRlid} holds for any operator, $\mathcal L$, of the form 
$\mathcal L = \mathbf B(x)\partial_x^2 
+ \mathbf M(x)$, irrespective of the particular form of   $\mathbf B(x)$ and $\mathbf M(x)$. Note however that $\textbf{K}$ is not differentiable at points on the diagonal. Furthermore, the kernel of the unsymmetrized Birman-Schwinger operator in Definition~\ref{def:BSoperators} is not continuous across the diagonal, and
hence is not Lipschitz continuous. 
\end{remark}

\begin{theorem}\label{Lipschitz} 
    Suppose that $\boldsymbol{\psi} \in C^1(\mathbb{R},\mathbb{R}^2)$,  that 
    $\boldsymbol{\psi}$ 
     and $\boldsymbol{\psi}_x$ decay exponentially, and that $\det(\mathbf M(x)) \neq 0$ for all $x\in \mathbb R$. Then the kernel $\mathbf{K}$  in \eqref{KformLip}  is globally Lipschitz-continuous on $\mathbb R\times \mathbb R$.
\end{theorem}

\begin{remark}
It is reasonable to ask why we need the pulse $\boldsymbol{\psi}$ to be as strong as $C^1$ 
in order to conclude that $\mathbf{K}$ is Lipschitz continuous and hence 
trace class.
In a nutshell, the reason is that if $f\in C^{0,\alpha}$ and $g\in C^{0,\beta}$
are H\"older continuous, then $f\circ g \in C^{0,\alpha\beta}$.
Consequently, if we were to make the weaker assumption that $\boldsymbol{\psi}\in C^{0,1}$ was only Lipschitz continuous, then $|\mathbf M| = \sqrt{\mathbf M^* \mathbf M} \in C^{0,1/2}$, since $\sqrt{\cdot} \in C^{0,1/2}$. Consequently, at best $\mathbf K\in C^{0,1/4}$, since it is defined in terms of $|\mathbf M|^{1/2}$. However, a counterexample of Bernstein~\cite{bernstein1934convergence} shows that 
to guarantee that a kernel is trace class it must be at least $C^{0,\alpha}$
where $\alpha > 1/2$.
\end{remark}

The proof of Theorem~\ref{Lipschitz} relies on the following general lemma.

\begin{lemma}\label{C1Lemma} Let $\mathbf M:\mathbb{R}\to\mathbb{C}^{k\times k}$ be $C^1$ in a vicinity of $x_0\in\mathbb{R}$ and $\det(\mathbf M(x_0))\neq0$. Then for $s=\pm\frac 12$ and $s=1$, the function $x\mapsto|\mathbf M(x)|^s$ is $C^1$ in a vicinity of $x_0$.
\end{lemma}

\begin{proof}
First, observe that $\mathrm{Spec}(\mathbf M^*(x)\mathbf M(x)) \subset [0,\infty)$ for all $x$ since $\mathbf M^*(x)\mathbf M(x)\ge0$. Furthermore,  
 $0\notin \mathrm{Spec}(\mathbf M^*(x_0)\mathbf M(x_0))$ as 
$\det(\mathbf M(x_0))\neq0$. Therefore, since $x\mapsto \mathbf M(x)$ is continuous
in the vicinity of $x_0$,
the upper semicontinuity of the spectrum~\cite[Remark IV.3.3]{Kato} implies that
there is a compact interval $[a,b] \subset (0,\infty)$ and an $\epsilon > 0$ so that for all $x \in (x_0-\epsilon, x_0+\epsilon)$, we have
$\mathrm{Spec}(\mathbf M^*(x)\mathbf M(x)) \subset [a,b]$.

Let $\gamma$ be a smooth contour located in the right half plane that surrounds the segment $[a,b]$. Since the matrix, $\mathbf M^*(x)\mathbf M(x)$, is diagonalizable, Cauchy's integral formula implies that for all $x \in (x_0-\epsilon, x_0+\epsilon)$,
\begin{equation}
|\mathbf M(x)|=\sqrt{(\mathbf M^*(x)\mathbf M(x))}=\frac1{2\pi i}\int_\gamma\sqrt{z}\big(\mathbf M^*(x)\mathbf M(x)-z\big)^{-1}\,dz.
\end{equation}
Here we have chosen the principal branch of $\sqrt{\cdot}$ which ensures that
$z\mapsto \sqrt{z}$ is complex analytic in the right half plane.
Finally, 
since $\gamma$ does not intersect $\mathrm{Spec}(\mathbf M^*(x)\mathbf M(x))$,
and since
$x\mapsto \mathbf M(x)$ is $C^1$ for $x\in (x_0-\epsilon, x_0+\epsilon)$, we conclude that $x\mapsto\big(\mathbf M^*(x) \mathbf M(x)-z\big)^{-1}$ is $C^1$ for all $z\in\gamma$. Differentiating under the integral, we conclude that $x\mapsto|\mathbf M(x)|$ is $C^1$. Similarly, $x\mapsto|\mathbf M(x)|^{\pm 1/2}$ 
is $C^1$.
\end{proof}

\begin{proof}[Proof of Theorem~\ref{Lipschitz}]
Recall that
\begin{equation}\label{eq:defKpm}
    \textbf{K}(x,y) = \begin{cases}
        \textbf{K}_-(x,y), \,\,\, &y \leq x,\\
        \textbf{K}_+(x,y), \,\,\, &y > x,
    \end{cases}
\end{equation}
where $\textbf{K}_{\pm}(x,y)$ are given in terms of $\mathbf Q$, $\textbf{R}_r(x), \textbf{R}_{\ell}(y)$,  and
 $e^{\mathbf A_\infty(x-y)}$ as in \eqref{KformLip}.
By Lemma~\ref{C1Lemma} and \eqref{eq:redefRrl}, $\textbf{R}_{\ell}(y)$ and $\textbf{R}_r(x)$ are $C^1$ in $\mathbb  R$, and hence $\mathbf K_\pm(x,y)$ is $C^1$ on $\mathbb R^2$.
    Fix $L>0$ and
     let $K(x,y) = K_{ij}(x,y)$ denote an entry of the matrix-valued kernel $\textbf{K}(x,y)$. 
    Then, $\exists M=M(L) > 0$ so that
    \begin{equation}\label{KLip}
        |{K}(x_1,y_1) - {K}(x_2,y_2)| \,\,\leq \,\, M \,\|(x_1, y_1) - (x_2, y_2)\|_2
    \end{equation}
for all pairs of points $P=(x_1,y_1)$ and $Q=(x_2,y_2)$ in $[-L,L]^2$
that lie on the same side of the diagonal $y=x$. 
To show that ${K}$ is Lipschitz-continuous on all of $[-L, L]^2$ we just need to show that \eqref{KLip} also holds when the points $P, Q$ lie on opposite sides of the diagonal. Let $C$ be the line segment connecting $P$ and $Q$ and let $R = (x, x)$ be the point of intersection of this line segment with the diagonal. 
Since the functions $K_\pm$ are also $C^1$ on the entire compact set $[-L, L]^2, \|\nabla K_\pm\|$ is bounded on $[-L, L]^2$,
 and so  $\exists M>0$ so that
\begin{equation}
    \|\nabla {K}(x,y)\| \leq M, \,\, \forall (x,y) \in [-L,L]^2 \text{ with } y \neq x.
\end{equation}
Now, by the fundamental theorem of calculus for line integrals, 
\begin{equation}
    {K}(x_1, y_1) = {K}_-(x,x) + \int_R^P \nabla {K} \cdot d\textbf{r}, 
\end{equation}
where $y_1 < x_1$ and ${K}_-(x,x)$ is the limit of $K(x,y)$ as $(x, y) \rightarrow (x, x)$ from the left. Similarly,
\begin{equation}
    {K}(x_2, y_2) = {K}_+(x, x) + \int_{R}^Q \nabla {K} \cdot d\textbf{r},
\end{equation}
where $y_2 > x_2$ and ${K}_+(x, x)$ is the limit of $K(x,y)$ as $(x, y) \rightarrow (x, x)$ from the right. 
We note that the limits $K_-(x,x)$ and $K_+(x,x)$ exist and are equal since, 
 by Proposition~\ref{prop:Kcts}, $K$ is continuous across the diagonal.
Therefore,
\begin{align} \nonumber 
    |{K}(x_1, y_1) - {K}(x_2, y_2)| &= \left| \int_R^P \nabla {K} \cdot d\textbf{r} - \int_R^Q \nabla {K} \cdot d\textbf{r} \right|\\ \nonumber 
    &\leq \left| \int_R^P \nabla {K} \cdot d\textbf{r} \right| + \left| \int_R^Q \nabla {K} \cdot d\textbf{r} \right|\\ \nonumber 
    &\leq M\left( \|(x_1, y_1) - (x, x)\|_2 + \| (x_2, y_2) - (x, x)\|_2\right)\\
    &= M \|(x_1, y_1) - (x_2, y_2)\|_2,
\end{align}
where the final equality holds since $R$ is on the line segment from $P$ to $Q$.
This argument shows that $K$ is Lipschitz-continuous and hence is absolutely continuous 
on  $[-L, L]^2$. Therefore,
$K$ is differentiable almost everywhere on $\mathbb R \times \mathbb R$.
Since $K$ is absolutely continuous on $[-L, L]^2$ for every $L>0$, we know that
for any $(x_1,y_1)$, $(x_2,y_2)\in \mathbb R\times \mathbb R$,
\begin{equation}
K(x_2,y_2) - K(x_1,y_1) = \int_C \nabla K\cdot d\textbf{r}.
\end{equation}
 where $C$ is the line segment joining $(x_1, y_1)$ to $(x_2, y_2)$.
Therefore, for each $L>0$, $\partial_x K$ and $\partial_y K$ are bounded above by $M=M(L)$ on $[-L,L]^2$.

Finally, we observe that since $\boldsymbol{\psi}$ is assumed to decay exponentially,
by Lemma~\ref{lemma:normKest} and Remark~\ref{rem:normKest}, $K$ and its partial derivatives 
decay exponentially.
Since we already know  that $\partial_x K$ and $\partial_y K$ are bounded on
$[-L,L]^2$ for any $L>0$, we conclude that $\partial_x K$ and $\partial_y K$ exist almost everywhere and are bounded 
on $\mathbb R \times \mathbb R$. Consequently, $K$ is globally Lipschitz-continuous on $\mathbb R \times \mathbb R$.
\end{proof}

Theorem~\ref{Lipschitz} applies quite generally to kernels of the form \eqref{KformLip}, where $\mathbf R_\ell$ and $\mathbf R_r$ are defined in terms
of a matrix-valued function $\mathbf M(x)$ as in \eqref{Rpertdef}. 
In the case of the CGLE, we recall from \eqref{MTilde} and \eqref{Rpertdef} that 
\begin{equation}\label{eq:M}
    \textbf{M} =  \textbf{N}_1|\boldsymbol{\psi}|^2 + \textbf{N}_2|\boldsymbol{\psi}|^4 + \left( 2\textbf{N}_1 + 4\textbf{N}_2|\boldsymbol{\psi}|^2\right)\boldsymbol{\psi} \boldsymbol{\psi}^T,
\end{equation}
where $|\boldsymbol{\psi}| := \|\boldsymbol{\psi}(x)\|_{2}$.
By explicitly calculating $\det(\mathbf M)$ we can derive conditions
on the parameters in the CGLE and on the amplitude of the pulse, $\boldsymbol\psi$,
which guarantee that the assumption, $\det(\mathbf M)>0$, in Theorem~\ref{Lipschitz}
holds. 
First,  we  observe that
\begin{equation}\label{Mrank1G}
    \textbf{M} = \textbf{G} + \textbf{H}\boldsymbol{\psi}\boldsymbol{\psi}^T,
\end{equation}
where
\begin{equation}\label{defG}
    \textbf{G} = |\boldsymbol{\psi}|^2 \begin{bmatrix}
    \alpha & -\beta\\ \beta & \alpha
    \end{bmatrix}
\qquad\text{and}\qquad
    \textbf{H} = \begin{bmatrix}
    4\alpha - 2\epsilon & -(4\beta - 2\gamma)\\
    4\beta - 2\gamma & 4\alpha - 2\epsilon
    \end{bmatrix},
\end{equation}
with 
$\alpha = \epsilon + \mu|\boldsymbol{\psi}|^2$ and $\beta = \gamma + \nu |\boldsymbol{\psi}|^2$. 
Now,
\begin{equation}
    \det \textbf{G} = |\boldsymbol{\psi}|^4(\alpha^2 + \beta^2) = |\boldsymbol{\psi}|^4 |\textbf{a} + |\boldsymbol{\psi}|^2 \textbf{b}|^2,
\end{equation}
where 
\begin{equation}\label{eqabdef}
   \textbf{a} = \begin{bmatrix}
   \epsilon \\ \gamma
   \end{bmatrix} 
   \qquad\text{and}\qquad
    \textbf{b} = \begin{bmatrix}
   \mu \\ \nu
   \end{bmatrix}. 
\end{equation}
Since $\textbf{M}$ is a rank one update of $\textbf{G},$ by the Sherman-Morrison formula \cite{Meyer}, we have that
\begin{equation}
    \det \textbf{M} = \det \textbf{G} [1 + \boldsymbol{\psi}^T \textbf{G}^{-1} \textbf{H} \boldsymbol{\psi}].
\end{equation}
A calculation shows that   
\begin{align} \nonumber \frac{\alpha^2 + \beta^2}{2}
\boldsymbol{\psi}^T \textbf{G}^{-1}\textbf{H}\boldsymbol{\psi} &=   \alpha^2 + \beta^2 + |\boldsymbol{\psi}|^2(\alpha \mu + \beta \nu)\\ \nonumber 
&=   |\textbf{a} + |\boldsymbol{\psi}|^2 \textbf{b}|^2  + |\boldsymbol{\psi}|^2(\textbf{a} + |\boldsymbol{\psi}|^2 \textbf{b}) \cdot \textbf{b}
=    ( \textbf{a} + |\boldsymbol{\psi}|^2 \textbf{b}) \cdot (\textbf{a} + 2|\boldsymbol{\psi}|^2 \textbf{b}),
\end{align}
and so, after further calculation we find that,
\begin{equation}\label{detMformula}
    \det \textbf{M} = |\boldsymbol{\psi}|^4 \left[ 3|\textbf{a} + |\boldsymbol{\psi}|^2\textbf{b}|^2 + 2(\textbf{a} + |\boldsymbol{\psi}|^2 \textbf{b}) \cdot |\boldsymbol{\psi}|^2 \textbf{b} \right].
\end{equation}

\begin{hypothesis}\label{hyp:lowroot} Suppose that 
\begin{enumerate}
\item 
$|\boldsymbol{\psi}(x)| > 0$ for all $x \in \mathbb{R}$,
\item $\textbf{a} = [\epsilon \,\, \gamma]^T$ and $\textbf{b} = [\mu \,\, \nu]^T$ are not both \textbf{0}, and
\item if $\textbf{a} \neq \textbf{0}$ and $\textbf{b} \neq \textbf{0}$ and
\begin{equation}\label{rminusdef}
    r_- := \frac{-4(\epsilon \mu + \gamma \nu) - \sqrt{16(\epsilon \mu + \gamma \nu)^2 - 15(\epsilon^2 + \gamma^2 )(\mu^2 + \nu^2)}}{5(\mu^2 + \nu^2)}
\end{equation}
 is real and positive, then
\begin{equation}\label{hyprminusassumption}
\max_{x \in \mathbb{R}}
    |\boldsymbol{\psi}(x)|^2 < r_-.
\end{equation}
\end{enumerate}
\end{hypothesis}

\begin{proposition}\label{prop:lowroot}
If Hypothesis \ref{hyp:lowroot} holds then
$\det(\textbf{M}(x)) > 0$ for all $x$.
\end{proposition}

\begin{remark}\label{RemSecantTraceClass} In the special case case of the hyperbolic secant solution of the NLSE, Hypothesis~\ref{hyp:lowroot} holds since
$\mathbf a = [0,\gamma]^T \neq \mathbf 0$ and
$\mathbf{b} = [\mu \,\, \nu]^T = \mathbf{0}$. 
\end{remark}

\begin{remark}
The proof below also shows that  the converse holds, provided $\boldsymbol{\psi}$
is continuous. 
\end{remark}

\begin{proof}
By (\ref{detMformula}),
 \begin{equation}\label{detMab}
    \det \textbf{M} 
  =  |\boldsymbol{\psi}|^4 \, Q(| \boldsymbol\psi|^2)
  \quad\text{where}\quad
  Q(t) = 3|\textbf{a}|^2 + 8\textbf{a} \cdot \textbf{b} \,t + 5|\textbf{b}|^2t^2.
  \end{equation}
By (\ref{detMab}) and item (1) in Hypothesis~\ref{hyp:lowroot}, it suffices to show that $Q(t) > 0$ for all 
$t \in \operatorname{Image}(|\boldsymbol\psi |^2) 
:= \{|\boldsymbol{\psi}(x)|^2\, :\, x \in \mathbb{R}\} \subset (0, \infty).$ If $\textbf{b} = \textbf{0}$, then $\textbf{a} \neq \textbf{0},$ and $Q(t) = 3|\textbf{a}|^2 > 0$ for all $t\in\mathbb R$.
If instead $\textbf{b} \neq \textbf{0}$ and $\mathbf a = \mathbf 0$, then
$Q(t) = 5|\textbf{b}|^2 t^2 > 0$ for all $t>0$.

Finally, we consider the case that
$\textbf{b} \neq \textbf{0}$ and $\mathbf a \neq \mathbf 0$.
First, we observe that since $Q(0) = 3|\textbf{a}|^2 > 0$ and $Q(t) \to \infty$ as $t \to \infty,$ either $Q$ has no positive roots, in which case $Q(t) > 0, \, \forall \, t > 0,$ or $Q$ has two positive roots. If $Q$ has two positive roots, the smaller of these, $r_{-},$ is given by (\ref{rminusdef}). Since $Q(t) > 0$ for all $t \in (0, r_-)$, we conclude from \eqref{hyprminusassumption} that
 $Q(|\boldsymbol{\psi}(x)|^2) > 0$ for all $x \in \mathbb{R}$, as required. 
\end{proof}

\section{Main Results}\label{sec:MainTheorems}

The main goal of this paper is to numerically compute the point spectrum,
$\sigma_{\rm pt}(\mathcal L)$, of the 
operator, $\mathcal L$, obtained by linearizing the CGLE~\eqref{cqcgle}
about a stationary pulse solution, $\psi$. 
The spectrum of $\mathcal L$ characterizes the linear stability of the pulse $\psi$. 
The standard approach to this
problem is to compute the zeros of the Evans function, $E(\lambda)$.
The main challenge with this approach is that the Evans function is defined in terms
of the Jost functions, whose numerical computation involves solving a stiff  system of differential equations. 
Here, we take an alternative approach in which 
 the point spectrum of $\mathcal L$
 is characterized as the zero set of the Fredholm determinant of a family of Birman-Schwinger integral operators, $\mathcal K(\lambda)$, 
parametrized by the spectral parameter, $\lambda\in\mathbb C\setminus \sigma_{\rm ess}(\mathcal L)$.

In this section, we apply the results of section~\ref{sec:Background} to the 
Birman-Schwinger operator, $\mathcal{K}(\lambda)$, given in Proposition~\ref{kktildej2}.
First, applying Theorem~\ref{TraceClassThmRealLine} and invoking Theorem~\ref{Lipschitz}, we identify conditions
under which $\mathcal{K}(\lambda)$ is trace class.  
In this situation, we can 
apply results of  Gesztesy, Latushkin, and Makarov~\cite{EJF}
to conclude that the regular Fredholm determinant, ${\det}_1(\mathcal I + \mathcal K(\lambda))$, is equal to the Evans function, $E(\lambda)$.
Second, applying Theorems~\ref{BackgroundTruncationTheorem} and 
\ref{quadphibound}, we  obtain  bounds on the errors between the Fredholm determinants, ${\det}_p
(\mathcal I + \mathcal K(\lambda))$, for $p=1,2$, and their numerical approximations,
$d_{p,Q}$, defined in \eqref{eq:dpQdef}.
In section~\ref{sec:NumResults}, we will apply the results in this section to
numerically compute $\sigma_{\rm pt}(\mathcal L)$. 

\begin{theorem}\label{TCthm}
Suppose that  $\lambda \notin \sigma_{\rm ess}(\mathcal L)$, that $\boldsymbol{\psi} \in C^1(\mathbb{R},\mathbb{C}^2)$,  that 
    $\boldsymbol{\psi}$ 
     and $\boldsymbol{\psi}_x$ decay exponentially, and that Hypotheses~\ref{hyp:BDL}  and \ref{hyp:lowroot} hold.
 Then $\mathcal{K}(\lambda) \in \mathcal{B}_1(L^2(\mathbb{R},\mathbb{C}^4))$ is trace class. Furthermore, 
 \begin{equation}\label{eq:det1eqE}
{\det}_1(\mathcal I + \mathcal K(\lambda)) = E(\lambda),
 \end{equation}
 and 
  \begin{equation}\label{eq:det2eqE}
{\det}_2(\mathcal I + \mathcal K(\lambda)) = e^{-\Tr(\mathcal{K}(\lambda))} E(\lambda).
 \end{equation}
\end{theorem}

\begin{proof}
Under the hypotheses, the kernel is Lipschitz-continuous by Theorem \ref{Lipschitz}.
Furthermore, by Lemma~\ref{lemma:normKest} and Remark~\ref{rem:normKest},
$\mathbf K(x,y)$ and its partial derivatives decay exponentially as 
$(x,y)$ moves away from the diagonal; that is, the estimate in  \eqref{KandDerivsexp} holds.  Therefore, 
 by Theorem \ref{TraceClassThmRealLine}, $\mathcal{K}$ is trace class.

Gesztesy, Latushkin and Makarov~\cite{EJF} prove a general result relating the
2-modified Fredholm determinant of a Birman-Schwinger operator
to the Evans function. Let
\begin{equation}\label{EofLambda}
    E(\lambda) = \det(Y_+(0; \lambda) + Y_-(0; \lambda)),
\end{equation}
where $Y_{\pm} = Y_{\pm}(x; \lambda)$ are the matrix-valued Jost solutions 
given in \cite[Definition 8.2]{EJF}.  By \cite[Theorem 9.4]{EJF},
$E(\lambda)$ is the Evans function for the stationary pulse $\psi.$
Their general result~\cite[Theorem 8.3]{EJF} is that
\begin{equation}\label{d2ThEvans}
    {\det}_2(\mathcal{I} + \mathcal{K}(\lambda)) = e^{\Theta(\lambda)} E(\lambda),
\end{equation}
where,
\begin{equation}\label{Thetadef}
    \Theta(\lambda) = \int_{0}^{\infty} \Tr(\mathbf{Q}(\lambda)\mathbf{R}(x))\,dx - \int_{-\infty}^0 \Tr((\mathbf{I}-\mathbf{Q}(\lambda))\mathbf{R}(x))\,dx
     = \int_{-\infty}^{\infty} \Tr(\mathbf{Q}(\lambda)\mathbf{R}(x))\,dx,
\end{equation}
where the final equality holds since $\Tr(\textbf{R}(x)) = 0$ by \eqref{Rpertdef}.

On the other hand,
since $\mathcal{K}(\lambda)$ is trace class,  $ \Tr(\mathcal{K}(\lambda))$ and $\det(\mathcal{I} + \mathcal{K}(\lambda))$ are defined and we also have that~\cite{Simon}
\begin{equation}\label{det2tracedet1}
    {\det}_2(\mathcal{I} + \mathcal{K}(\lambda)) = e^{-\Tr(\mathcal{K}(\lambda))} \det(\mathcal{I} + \mathcal{K}(\lambda)).
\end{equation}
Therefore, to prove~\eqref{eq:det1eqE} it suffices to show that $\Tr(\mathcal{K}(\lambda)) = - \Theta(\lambda)$.
Now by~\cite{GZL2025NumericalFredholm}, \eqref{opkernel}, and the fact that
$\Tr(\mathbf A\mathbf B) = \Tr(\mathbf B\mathbf A)$, we have that
\begin{equation}
    \Tr(\mathcal{K}) = \int_{-\infty}^{\infty} \Tr(\mathbf{K}(x,x))\, dx
    = -\int_{-\infty}^{\infty} \Tr(\mathbf{R}_r(x)\mathbf{Q}\mathbf{R}_{\ell}(x))\, dx
    = -\int_{-\infty}^{\infty} \Tr (\mathbf{Q}\mathbf{R}(x)) = -\Theta,
\end{equation} 
as required.
\end{proof}

\begin{remark}
Theorem~\ref{LptKtilde} and Equation (\ref{d2ThEvans}) show that the point spectrum of the linearized operator, $\mathcal L$, can be computed either by finding the zeros of the Evans function, $E(\lambda)$, or the zeros of the $2$-modified Fredholm determinant,
 ${\det}_2(\mathcal{I} + \mathcal{K}(\lambda))$.
However, although it is generally thought to be true, 
in this paper we do not address the question of whether 
the multiplicities of the eigenvalues of $\mathcal L$ equal the 
multiplicities of the zeros of ${\det}_2(\mathcal{I} + \mathcal{K}(\lambda))$ and 
$E(\lambda)$. 
\end{remark}
    
The following result now follows immediately from Theorems \ref{truncthm}, \ref{quadphibound}, \ref{Lipschitz} and \ref{TCthm}.

\begin{theorem}\label{errorbds}
Suppose that $\psi = \psi(x)$ is a stationary pulse solution of the CQ-CGLE which satisfies the assumptions of Theorem~\ref{TCthm}.  Then for $p=1,2$, 
\begin{equation}\label{total2error}
    \left| {\det}_p(\mathcal{I} + \mathcal{K}(\lambda)) - d_{p,Q}(\lambda) \right| \leq e^{-aL}\boldsymbol{\Phi}\left(\frac{8C}{a} \right) + \frac{\sqrt{2 \pi} e}{8}\Delta x \boldsymbol{\Phi}(8L\|\textbf{K}\|_{W^{1,\infty}}),
\end{equation}
where the quadrature approximation, $d_{p,Q}(\lambda)$, is obtained by truncating the operator, $\mathcal K(\lambda)$, to a finite interval $[-L,L]$ which is then discretized with a step size of $\Delta x$ and where
\begin{equation}
   \|\textbf{K}\|_{W^{1, \infty}} = \max\{ \|\partial_x \textbf{K}\|_{L^{\infty}([-L,L]^2, \mathbb{C}^{k \times k})}, \|\partial_y \textbf{K}\|_{L^{\infty}([-L,L]^2, \mathbb{C}^{k \times k})}, \|\textbf{K}\|_{L^{\infty}([-L,L]^2, \mathbb{C}^{k \times k})}\}.
\end{equation}
\end{theorem}

\section{Numerical Results}\label{sec:NumResults}

In this section, we present the results of numerical simulations which validate the theory developed in section \ref{sec:MainTheorems}.
First, we numerically compute the regular and 2-modified Fredholm determinants
of the Birman-Schwinger operator, $\mathcal K(\lambda)$, for the  hyperbolic secant solution of the NLSE. We validate these computations using an analytical formula for the Evans function 
of the hyperbolic secant pulse given in Appendix~\ref{App:EvansSech} and  a general formula for $\Tr(\mathcal K(\lambda))$
given in Appendix~\ref{App:TraceK}.
Second, we compute the point spectrum for a numerically computed pulse solution of the CGLE by using a  root-finding method to 
determine the zeros of the approximated regular
Fredholm determinant of $\mathcal K(\lambda)$.
We validate this computation by comparison to results obtained by
Shen  et al.~\cite{shen2016spectra} using a method that is similar in spirit to numerical Evans function
methods, but which involves the iterative solution of a nonlinear eigenproblem.

By Remark~\ref{RemSecantTraceClass}, the 
Birman-Schwinger operator, $\mathcal K(\lambda)$, for the hyperbolic secant solution of the NLSE is trace class. 
In Theorem~\ref{EvansThm}, we provide a formula for the Evans function, $E(\lambda)$, of this pulse
and in \eqref{AppB:TraceKNLSE} we provide a formula for $\Tr(\mathcal K(\lambda))$.
Consequently, we can validate the numerical computation of the
regular and 2-modified Fredholm determinants of $\mathcal K(\lambda)$
with the aid of \eqref{eq:det1eqE} and \eqref{eq:det2eqE} in Theorem~\ref{TCthm}.
The Evans function has a zero of multiplicity four at $\lambda=0$ and is
finite at the edges of the essential spectrum at $\lambda = \pm i/2$ (see Remark~\ref{AppARem:EvansEss}). 
In \cite{GZL2025NumericalFredholm}, we validated the convergence result in 
Theorem~\ref{errorbds} at $\lambda=0$ by studying how the Fredholm determinant
converges to zero as the width $L$ of the simulation window increases and the
step size, $\Delta x$, of the discretization of that window decreases.

In Figure~\ref{fig:Det1Det2Sech}, we compare the numerically approximated regular and 2-modified Fredholm determinants 
to  the analytical formulae given by the right-hand sides of
\eqref{eq:det1eqE} and \eqref{eq:det2eqE}. In the left panel, we show the results
when $\lambda$ is varied from $-1$ to $1$ along the real axis.
These results were obtained with $L\approx7.32$ and $\Delta x = 0.0586$ ($N= 251$ points).
In the right panel, by zooming in near $\lambda=0$, we 
illustrate the nature of the convergence 
of the approximate regular Fredholm determinant 
to the analytical Evans function
as $\Delta x$ decreases. The restriction of the 
absolute value of the 
approximate Fredholm determinant to the real axis
has two symmetrically positioned minima whose distance
from $\lambda=0$ halves each time $\Delta x$ is halved. 
We note that a minimum of $|\det_1(\mathcal I + \mathcal K(\lambda))|$ corresponds to a zero of $\det_1(\mathcal I + \mathcal K(\lambda))$.

In the left panel of Figure~\ref{fig:SechAndYannan}
we  show the the results when $\lambda$ is varied from the origin to the edge of the essential spectrum along the imaginary axis.  
We note that $|e^{-\Tr(\mathcal K(\lambda))}|>1$ since, by \eqref{AppB:TraceKNLSE},
the real part of 
$\Tr(\mathcal K(\lambda))$ is negative for $\lambda \notin \sigma_{\rm ess}(\mathcal L)$. Consequently, $|\det_2(\cI + \cK(\lambda)) | > |\det_1(\cI + \cK(\lambda)) |$. Moreover, since $\Re(\Tr(\mathcal K(\lambda)))\to -\infty$
as $\lambda \to \pm i/2$, the 2-modified determinant blows up at the edges of the essential spectrum. 
Even though the trace (and hence the trace class norm~\cite{Simon}) blows up, it is significant that
the Evans function, and hence the regular Fredholm determinant, remain bounded as $\lambda\to \pm i/2$.
This fact could not be predicted from the general theory of Fredholm determinants, which only gives a bound
for $\det_1$ in terms of the trace class norm~\cite{Simon}.
Finally, we observe that
the agreement  between the numerical 
approximations and the analytical formulae is excellent, both at the order four zero at $\lambda=0$ (Figure~\ref{fig:Det1Det2Sech}) 
and at the singularity of the 2-modified determinant at $\lambda = i/2$ (Figure~\ref{fig:SechAndYannan}, left).

\begin{figure}[!thb]
    \centering
    \includegraphics[width=.485\linewidth]{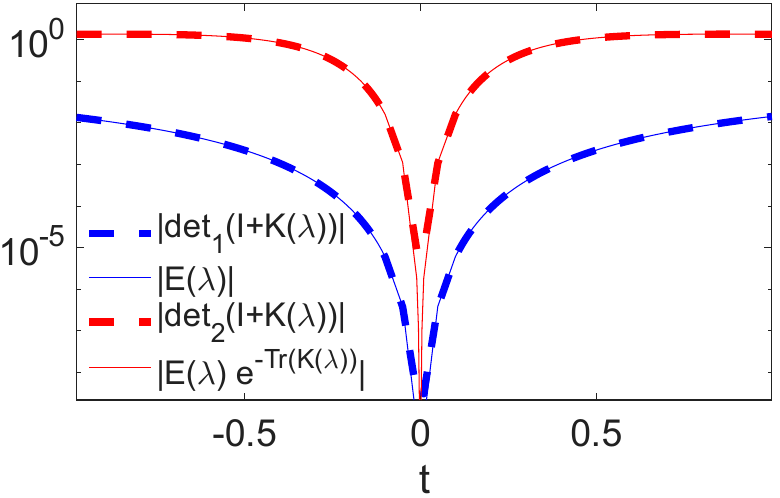}
     \includegraphics[width=0.5
     \linewidth]{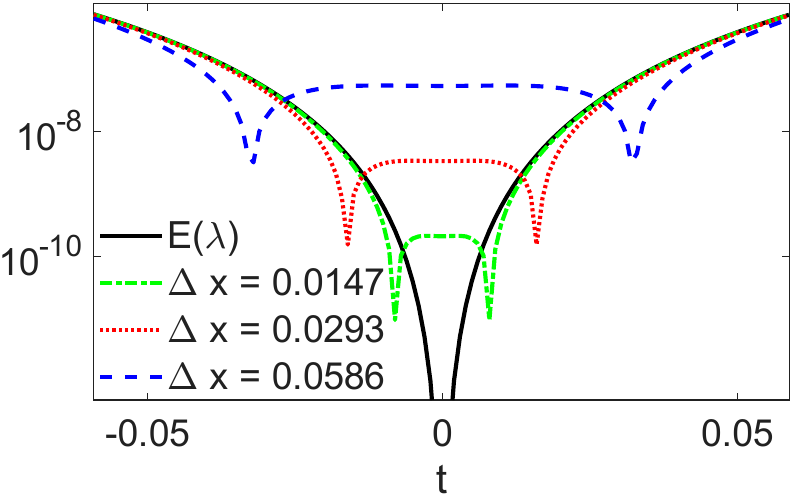}
    \caption{{\bf Left:} The numerically-approximated regular and $2$-modified Fredholm determinants for the sech solution of the NLSE, as compared to 
    analytical formulae given in terms of the Evans function, evaluated along the line $\lambda(t) = t$. {\bf Right:} 
    Zoomed in version of results on the left in which the 
    regular Fredholm determinant is computed with finer discretizations,
    $\Delta x$, across the pulse} 
    \label{fig:Det1Det2Sech}
\end{figure}

For our second validation, we compare the point spectrum for a stationary CGLE pulse computed from the zero set of $\det_1(\cI + \cK(\lambda))$ with a numerical result obtained 
by Shen et al.~\cite{shen2016spectra}. For this result, we chose the
parameters in the CGLE~\eqref{cqcgle} to be 
$D = -0.001$, $\gamma = 1$, $\nu= 10$, $\delta = -0.01$, $\epsilon = 1$,
$\beta = 0.08$, and  $\mu = -3$. 
Starting from an initial Gaussian pulse, 
we used the Fourier split-step method~\cite{sinkin2003optimization} with a wide spatial window  to obtain a stationary pulse. 
We observed that the numerically computed stationary pulse, $\psi$, (not shown) appears to  
be $C^1$, to decay exponentially, and to be nonzero. In addition, the 
quantity, $r_-$, in Hypothesis~\ref{hyp:lowroot} is not real. 
 Therefore, by Proposition~\ref{prop:lowroot} and Theorem~\ref{TCthm}, the operator, $\mathcal K(\lambda)$, is trace class. 
For the computation 
of the Fredholm determinant, we truncated the spatial window to $L\approx 7.32$, and used 
$\Delta x = 0.0586$ ($N=251$ points), as before.
Then we used a method of Lyness and Delves~\cite{delves1967numerical,lyness1967numerical} to locate the roots
of the complex analytical function, $\lambda \mapsto d_{1,Q}(\lambda)$, which approximates $\det_1(\cI + \cK(\lambda))$. 
With this method, contour integrals are first used to count zeros (with multiplicity)
within given small regions, and then a version of 
Newton's method for complex analytic functions is used to
determine the precise locations of these zeros. 
In the right panel of Figure~\ref{fig:SechAndYannan}, we show the essential spectrum (blue lines) and 
eigenvalues as computed from the regular Fredholm determinant (red pluses) and by Shen et al.~\cite{shen2016spectra} (blue circles). The agreement between the two completely different numerical methods is excellent.
We observe that the two complex eigenvalues are just barely stable, lying very close to the imaginary axis. Nevertheless, we are able to locate all the eigenvalues accurately, even in this situation. 
The two roots found by Lyness method near $\lambda=0$ were located on the negative real axis at $\lambda= -0.1230\times 10^{-4}$ and 
$\lambda= -0.3826\times 10^{-4}$. The value of the Fredholm determinant at these two numerically computed roots was on the order of $10^{-17}$
while the value at $\lambda = 0$ was on the order of $10^{-9}$.
We observed  similar trends when Lyness method was applied to the
hyperbolic secant pulse. 
The values of  the Fredholm determinant at the  two complex eigenvalues near the imaginary axis and  the other two real eigenvalues  were on the  order of $10^{-13}$ and $10^{-15}$, respectively. 

\begin{figure}[!thb]
    \centering
    \includegraphics[width=0.49\linewidth]{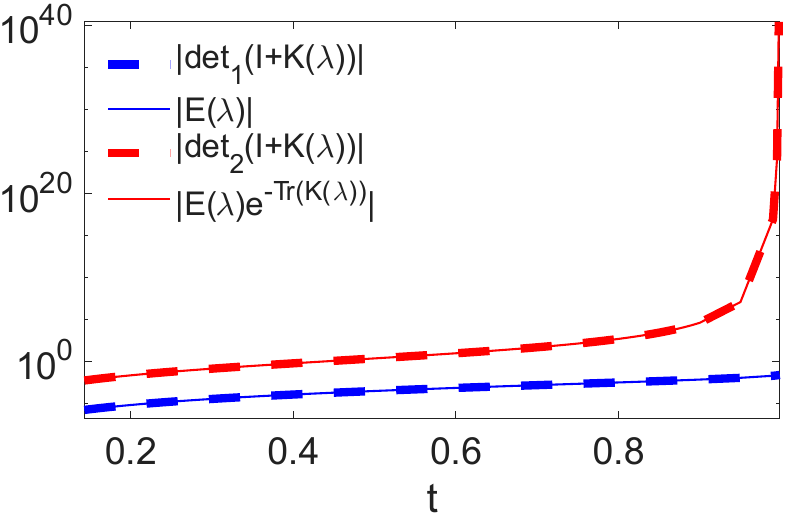}
       \includegraphics[width=0.49\linewidth]
       {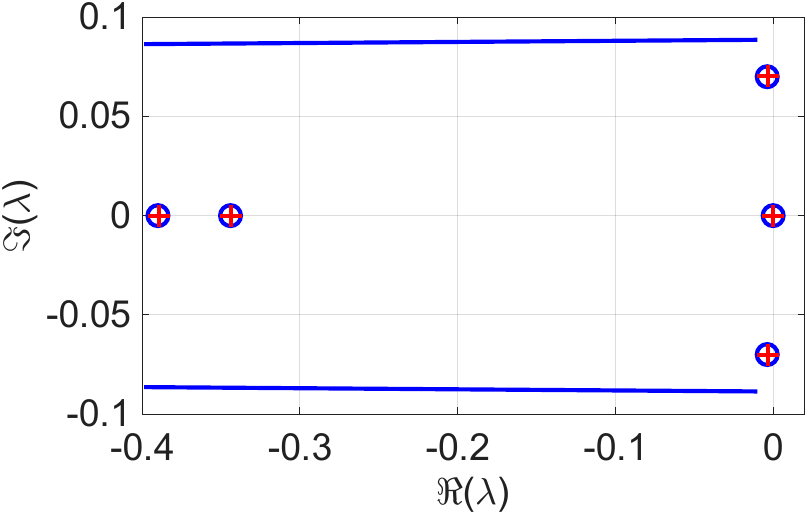}
    \caption{{\bf Left:} 
     Behavior of the regular and $2-$modified Fredholm determinants, 
     for the
sech solution of the NLSE, 
      compared to analytical formulae given in terms of the Evans function, as
     $\lambda$ approaches the edge of $\sigma_{{\rm ess}}(\mathcal{L})$ along the line $\lambda (t) = it/2$.
        {\bf Right:}  
Essential spectrum (solid blue lines) and eigenvalues, 
for a  stationary pulse solution of the 
CGLE, computed using the approximate regular Fredholm determinant (red pluses) compared with results obtained by Shen et al.~\cite{shen2016spectra} (blue circles)}
    \label{fig:SechAndYannan}
\end{figure}

In the left panel of Figure~\ref{fig:YanDetsToEdge}, we plot the real and imaginary parts of the regular Fredholm determinant along the line joining the 
eigenvalue at zero to the eigenvalue, $\lambda_c$,  near the edge of the essential spectrum. 
In the right panel, we plot the regular and 2-modified Fredholm determinants
along the line from $\lambda_c$ to the edge, $\lambda_e$, 
of the essential spectrum. We observe
that the regular Fredholm determinant remains bounded, while the 
2-modified determinant blows up as $\lambda \to \lambda_e$. 
This phenomenon was already presaged in Remark~\ref{rem:condP} on the blow up of $\operatorname{cond}(\mathbf P(\lambda))$.
To provide  a more complete analysis, in Theorem~\ref{AppB:TraceKThm} we prove that for any pulse solution of the CGLE for which $\mathcal K(\lambda)$ is trace class, there is a path $\lambda=\lambda(t)$ converging to the
edge of the essential spectrum along which 
$\Re(\Tr\,\mathcal K(\lambda(t)))\to -\infty$ and hence
$|e^{-\Tr\,\mathcal K(\lambda(t))}|\to\infty$. 
Since $\Tr (\mathcal K)  \leq \|\mathcal K\|_{\mathcal B_1}$
is dominated by the trace class norm~\cite{Simon}, we conclude that
$\|\mathcal K(\lambda(t))\|_{\mathcal B_1}\to\infty$ as well. 
Consequently, the 2-modified Fredholm determinant must also blow up, since 
even if $\det_1(\mathcal I + \mathcal K(\lambda))$ were to have a zero at $\lambda=\lambda_e$, this zero would
have finite order and so could not cancel out the
essential singularity in $e^{-\Tr\,\mathcal K(\lambda(t))}$.
On the other hand,
the numerics convincingly show that $\det_1(\mathcal I + \mathcal K(\lambda(t)))$ remains bounded as $\lambda\to\lambda_e$, just as in the case of the NLSE. 
However, in this paper we do not theoretically address  whether the regular Fredholm determinant (or equivalently, the Evans function) remains bounded for a general CGLE pulse.
The extremely rapid blow up of the 2-modified Fredholm determinant
along paths from eigenvalues, $\lambda_c, $ near the edge of the essential spectrum to $\lambda_e$, 
suggests that it is not practical to compute such eigenvalues
using the 2-modified determinant.

\begin{figure}[!thb]
    \centering
  \includegraphics[width=0.49\linewidth]{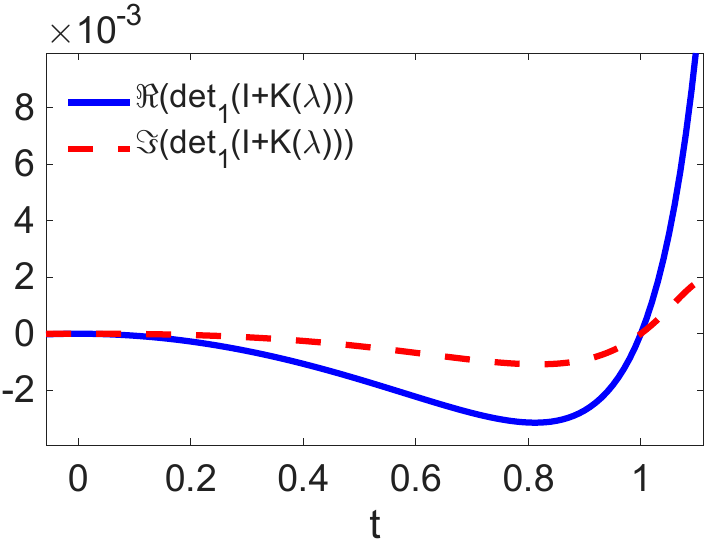}
       \includegraphics[width=0.49\linewidth]{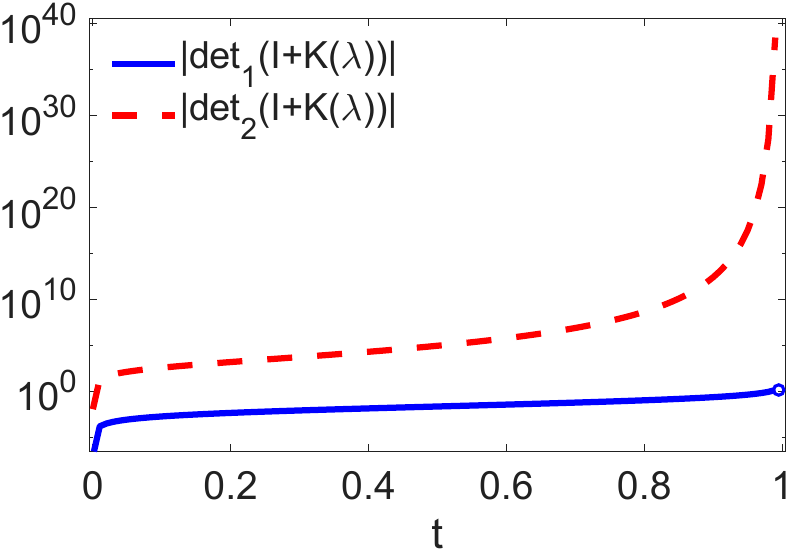}
    \caption{
    {\bf Left:} Approximate regular Fredholm determinant,
    for the stationary pulse solution of the 
CGLE, along the line $\lambda(t) = \lambda_ct,$ between eigenvalue $\lambda = 0$ and complex eigenvalue $\lambda_c \approx -0.0033 + 0.0704i$. 
    {\bf Right:} Behavior of the regular (dashed red line) and $2-$modified (solid blue line) Fredholm determinants as $\lambda$ approaches the edge, $\lambda_{{e}}$, of the essential spectrum along the line $\lambda(t) = \lambda_c + (\lambda_{{e}} - \lambda_c)t$
    }
    \label{fig:YanDetsToEdge}
\end{figure}

\appendix

\section{The Evans function for the hyperbolic secant pulse}\label{App:EvansSech}

In this appendix, we derive a formula for the Evans function of the 
hyperbolic secant solution of the NLSE using the definition  given by Gesztesy, Latushkin and Makarov~\cite{EJF}.
The calculation is related to the one in Kapitula and Promislow~\cite[Sec.\ 10.4.1]{Kap}.
However, as we show in Remark~\ref{KapEJFNorm} below, the Jost solutions in \cite{Kap} have a different $\lambda$-dependent normalization from those in~\cite{EJF}. 
The normalization 
in \cite{EJF} is the unique one for which the Evans function is equal
to the regular Fredholm determinant of the trace class Birman-Schwinger 
operator, $\mathcal K(\lambda)$. Due to these differing normalizations, 
the Evans function in \cite{Kap} converges to zero at the edge of the essential
spectrum, while the Evans function we use converges to a nonzero finite value.
We now state our main result. 

\begin{theorem}\label{EvansThm}
Let $\psi(t,x) = \sech(x) $ be the hyperbolic  secant solution of the NLSE
\begin{equation}\label{yuripde}
        i \partial_t \psi + \tfrac 12 (\partial_x^2 - 1) \psi + |\psi|^2 \psi = 0,
    \end{equation}
    and let $\mathcal L$ be the differential operator obtaining by linearizing
    \eqref{yuripde} about $\psi$, which is given by \eqref{cq3}
with $D=1$, $\gamma=1$, $\alpha=-\tfrac 12$ and all other parameters set to zero.
Then the Evans function of the eigen-problem $\mathcal L \mathbf p = \lambda \mathbf p$ is given by
 \begin{equation}\label{eq:EvansSech}
        E(\lambda) := \frac{-16 \lambda^4}{\left[1 + \sqrt{1 - 2i \lambda}\right]^4 \,\left[1 + \sqrt{1 + 2i \lambda}\right]^4},
\end{equation}
where $\sqrt{\cdot}$ denotes the principal branch of the square root.
\end{theorem}

\begin{remark}\label{AppARem:EvansEss}
The essential spectrum, $\sigma_{\rm ess}(\mathcal L)$, is the union of the half-lines, $\pm i[1/2, \infty)$,
on the imaginary axis. Since $\sqrt{z}$ always has a nonnegative real part 
the denominator in \eqref{eq:EvansSech} is never zero. Therefore
the Evans function in \eqref{eq:EvansSech} is defined  on the entire complex plane, is complex analytic on $\mathbb C \setminus \sigma_{\rm ess}(\mathcal L)$
and is discontinuous across $\sigma_{\rm ess}(\mathcal L)$. In addition, $E(-\lambda)
= E(\lambda)$ for all $\lambda \in \mathbb C \setminus \sigma_{\rm ess}(\mathcal L)$.   
\end{remark}

\begin{proof}[Proof of Theorem~\ref{EvansThm}]

The eigen-problem is given by
\begin{equation}\label{eq:AppOrigEP}
0\,\,=\,\,(\mathcal L - \widehat\lambda) \widehat{\mathbf p} \,\,=\,\,
\left( \tfrac 12 (\partial_x^2 - \mathbf I)\mathbf J -  \widehat\lambda \mathbf I 
+ \sech^2(x) \begin{bmatrix}0 &-1\\3&0\end{bmatrix}\right) \widehat{\mathbf p},
\end{equation}
where we use the notation $\widehat\lambda$ in place of $\lambda$ to simplify
notation below and where
$\mathbf J = \begin{bmatrix} 0&-1\\1&0\end{bmatrix}$.
We calculate the Evans function for an 
alternate form of the linearization of the NLSE, which we then
 transform back into the 
original formulation to obtain \eqref{eq:EvansSech}. 
The alternate form is given by linearizing the NLSE about a
complex vector-valued solution, $\boldsymbol{\phi} = [\psi \, \overline{\psi}]^T$, whereas the 
original form is given by linearizing about the
real vector-valued solution, $\boldsymbol{\psi} = [\Re(\psi) \, \Im(\psi)]^T$.
These two formulations  are related by $\boldsymbol{\psi} = \textbf{P} \boldsymbol{\phi}$ for the unitary matrix
\begin{equation}
\mathbf P \,\,=\,\,\frac 1{\sqrt{2}}
\begin{bmatrix}
1&1\\-i&i
\end{bmatrix}.
\end{equation}

The alternate form of the linearization is obtained by multiplying \eqref{eq:AppOrigEP} on the left by $i\sigma_3 \mathbf P^*$ and 
making the substitution $\widehat{\mathbf p} = \mathbf P \mathbf p$, to obtain
the equivalent eigen-problem  in \cite[Sec.\ 10.4.1]{Kap} given by
\begin{equation}\label{cLevp}
\left[\partial_x^2-\bfI_{2\times2}-\lambda\boldsymbol\sigma_3+2\sech^2(x)\boldsymbol\sigma\right]
\mathbf p=0, \qquad\text{where } \lambda = -2i\widehat\lambda.
\end{equation}
Here and below we use the  notation,
\begin{align}\label{notsigma}
\boldsymbol\sigma_3=\begin{bmatrix}1&0\\0&-1\end{bmatrix}, \, 
\boldsymbol\sigma_2=i \begin{bmatrix}0&-1\\1&0\end{bmatrix},\,
\boldsymbol\sigma=\begin{bmatrix}2&1\\1&2\end{bmatrix},\,
\mathbf u=\begin{bmatrix}1\\0\end{bmatrix}, \,
\mathbf v=\begin{bmatrix}0\\1\end{bmatrix},\,
\mathbf w=\begin{bmatrix}1\\1\end{bmatrix}.
\end{align}

Converting \eqref{eq:AppOrigEP} and \eqref{cLevp} to $4\times 4$ systems of first-order equations, $\partial_x \widehat{\mathbf Y} = \left[ \widehat{\mathbf A}_\infty(\widehat\lambda)
+ \widehat{\mathbf R}(x) \right] \widehat{\mathbf Y}$
and $\partial_x {\mathbf Y} = \left[ {\mathbf A}_\infty(\lambda)
+ {\mathbf R}(x) \right] {\mathbf Y}$, we find that 
$\widehat{\mathbf Y} =  \widetilde{\mathbf P}{\mathbf Y}$,
where $\widetilde{\mathbf P} = \begin{bmatrix} \mathbf P & \mathbf 0 \\ \mathbf 0 & \mathbf P \end{bmatrix}$.
Consequently, the generalized matrix-valued Jost solutions in \cite[Definition 7.2]{EJF} for these two problems
are also 
related by $\mathbf Y_\pm(x) =  \widetilde{\mathbf P}\widehat{\mathbf Y}_\pm(x)$.
Therefore, letting $\mathbf Y(x) = \mathbf Y_+(x) +  \mathbf Y_-(x)$, 
the two Evans functions defined as in \cite[Definition 7.5]{EJF} are related by
$\widehat E(\widehat\lambda) = \det( \widehat{\mathbf Y}(0))
= \det( \widetilde{\mathbf P}) \det({\mathbf Y}(0))
= - E(\lambda)$. Theorem~\ref{EvansThm} now follows from \eqref{cLevp} and Proposition~\ref{AltEvansThm} below.  
\end{proof}

\begin{proposition}\label{AltEvansThm} The Evans function 
for the eigen-problem \eqref{cLevp} is given by 
\begin{equation}\label{eff}
E(\lambda)=\frac{\lambda^4}{(1+\mu(\lambda))^4(1+\nu(\lambda))^4},
\end{equation}
where
\begin{equation}\label{defmunu}
 \mu=\sqrt{1-\lambda} \qquad\text{and}\qquad  \nu=\sqrt{1+\lambda}.
\end{equation}
\end{proposition}

\begin{proof}
Throughout, we assume that $\lambda\in\C\setminus(-\infty,-1]$ and 
$\lambda\in\C\setminus[1,+\infty)$,
which ensures that 
$\Re \nu>0$ and  $\Re \mu>0$.
In addition, for definiteness, we will often assume that 
\begin{equation}\label{remunuin}
0 < \Re \nu<\Re \mu
\end{equation}
(which holds, say, when $\lambda\in(-1,0)$), but the cases when $\Re\nu=\Re\mu$ or $\Re\nu>\Re\mu$ could be treated similarly.

First, we find two linearly independent solutions, $\bfp$ and $\bfq$, of \eqref{cLevp} that are exponentially decaying as $x\to +\infty$. For this, we substitute into \eqref{cLevp} the expressions
\begin{equation}
    \rexp^{-\mu x}\left[(A+B\tanh x)\mathbf v+C(\sech^2x)\mathbf w\right] \text{ and }
\rexp^{-\nu x}\left[(A+B\tanh x)\mathbf u+C(\sech^2x)\mathbf w\right],
\end{equation}
 cf.\ notation in \eqref{notsigma}, and use the method of undetermined coefficients to find $A,B,C\in\C$. For the first and second expressions  
 the coefficients satisfy the systems,
 \begin{equation}\begin{cases}2B+4\mu C=0\\ 4C+2A+C(\mu^2-\nu^2)=0\end{cases}
 \qquad\text{and}\qquad
 \, \begin{cases}4C-2\mu B+4A=0\\ 2B+4\mu C=0,\end{cases} \end{equation}
 respectively, which yields the Jost solutions
 \begin{align}\label{bfp}
 \bfp(x)&=\rexp^{-\mu x}\left[(\lambda-2-2\mu\tanh x)\mathbf v+(\sech^2x)\mathbf w\right],\\
 \label{bfq}
 \bfq(x)&=\rexp^{-\nu x}\left[(-\lambda-2-2\nu\tanh x)\mathbf u+(\sech^2x)\mathbf w\right].
 \end{align}

\begin{remark}\label{KapEJFNorm}
 The Jost solutions $\bfp$ and $\bfq$ in \eqref{bfp} and \eqref{bfq} are related to the solutions obtained by Kaup in \cite{Kaup} by
 \begin{equation}
 \bfp(x)=-(1+\mu(\lambda))^2\bfp_\ell(x,\lambda) \qquad\text{and}\qquad
 \bfq(x)=-(1+\nu(\lambda))^2\bfp_r(x,\lambda),
 \end{equation}
where $\bfp_\ell(x,\lambda)$ and $\bfp_r(x,\lambda)$ in given in \cite[pp.\ 330-331]{Kap}. 
\end{remark}
 
 Clearly, 
 \begin{align}\label{bfpp}
 \bfp'(x)&=-\mu\bfp(x)+\rexp^{-\mu x}\big((-2\mu\sech^2 x)\bfv-(2\sech^2x\tanh x)\bfw\big),\\
 \label{bfqp}
 \bfq'(x)&=-\nu\bfq(x)+\rexp^{-\nu x}\big((-2\nu\sech^2 x)\bfu-(2\sech^2x\tanh  x)\bfw\big),
 \end{align}
 and for future use we note that,
 \begin{align}\label{detpzero}
 \det\begin{bmatrix}\bfq(0)&\bfp(0)\end{bmatrix}&=\det\begin{bmatrix}-\lambda-1&1\\1&\lambda-1\end{bmatrix}=-\lambda^2,\\
  \det\begin{bmatrix}\bfq'(0)&\bfp'(0)\end{bmatrix}&=\mu\nu
  \det\begin{bmatrix}-\lambda+1&1\\1&\lambda+1\end{bmatrix}=-\mu\nu\lambda^2.
 \label{detqzero}
 \end{align}
 
 To define the Evans function $E=E(\lambda)$, we will follow \cite{EJF} and study the first order ODE system corresponding to the eigenvalue problem \eqref{cLevp},
 \begin{equation}\label{ODE}\begin{split}
 \bfy'&=(\bfA_0(\lambda)+\bfR(x)\big)\bfy, \,\,
 \bfy=\begin{bmatrix}\bfp\\\bfp'\end{bmatrix},\, x\in\R,\\
\bfA_0(\lambda)&:=\begin{bmatrix}0&\bfI_{2\times2}\\\bfI_{2\times2}+\lambda\boldsymbol\sigma_3&0\end{bmatrix},\,\,
 \bfR(x):=\begin{bmatrix}0&0\\-2\sech^2(x)\boldsymbol\sigma&0\end{bmatrix}.\end{split}
 \end{equation}
We begin with the asymptotic system $\bfy'=\bfA_0(\lambda)\bfy$. The identity \[(\bfA_0(\lambda))^2
=\diag\{\bfI_{2\times2}+\lambda\boldsymbol\sigma_3,\bfI_{2\times2}+\lambda\boldsymbol\sigma_3\}\]
and assertion ${\rm Spec } (\bfI_{2\times 2}+\lambda\boldsymbol\sigma_3)=\{\nu^2,\mu^2\}$
show that the eigenvalues of $\bfA_0(\lambda)$ are given by
\begin{equation}\label{evA}
{\rm Spec } (\bfA_0(\lambda))=\big\{-\mu,-\nu,\nu,\mu\}
\end{equation}
with the respective eigenvectors, 
\begin{equation}\label{eivecA}
\bfy_1=\begin{bmatrix}\bfv\\-\mu \bfv\end{bmatrix},\,
\bfy_2=\begin{bmatrix}\bfu\\-\nu \bfu\end{bmatrix},\,
\bfy_3=\begin{bmatrix}\bfu\\\nu \bfu\end{bmatrix},\,
\bfy_4=\begin{bmatrix}\bfv\\\mu \bfv\end{bmatrix}.\,
\end{equation}
The projection $\bfQ_1$  onto ${\rm Span} \{\bfy_1\}$ parallel to ${\rm Span }\{\bfy_2,\bfy_3,\bfy_4\}$
 and the projection $\bfQ_2$  onto ${\rm Span} \{\bfy_2\}$ parallel to ${\rm Span }\{\bfy_1,\bfy_3,\bfy_4\}$
 are given by 
 \begin{equation}\label{projq1q2}
 \bfQ_1=\begin{bmatrix}0&0&0&0\\0&\frac12&0&-\frac1{2\mu}\\
 0&0&0&0\\0&-\frac{\mu}2&0&\frac12\end{bmatrix}
 \qquad\text{and}\qquad 
  \bfQ_2=\begin{bmatrix}\frac12&0&-\frac1{2\nu}&0\\0&0&0&0\\
 -\frac{\nu}2&0&\frac12&0\\0&0&0&0\end{bmatrix},
 \end{equation}
while $\bfQ_1+\bfQ_2$ is  the projection onto 
${\rm Span }\{\bfy_1,\bfy_2\}$ parallel to 
${\rm Span }\{\bfy_3,\bfy_4\}$ 
Similar formulas without any minus signs hold for $\bfQ_3$ and $\bfQ_4$.

We then directly verify that the fundamental matrix solution $\Phi$ satisfying $\Phi'=\bfA_0(\lambda)\Phi$, $\Phi(0)=\bfI_{4\times4}$ is given by 
\begin{equation}\label{phi}
\Phi(x)=\begin{bmatrix}\cosh(\nu x)&0&\frac1\nu\sinh(\nu x)&0\\
0&\cosh(\mu x)&0&\frac1\mu\sinh(\mu x)\\
\nu\sinh(\nu x)&0&\cosh(\nu x)&0\\
0&\mu\sinh(\mu x)&0&\cosh(\mu x)\end{bmatrix}
\end{equation}
The Bohl (or, equvalently, Lyapunov) exponents, $\varkappa(\bfQ_j)=\varkappa'(\bfQ_j)$,
for the asymptotic system are given by
\begin{equation}\label{bohls}
\varkappa(\bfQ_1)=-\Re \mu,\qquad \varkappa(\bfQ_2)=-\Re \nu,\qquad \varkappa(\bfQ_3)=\Re \nu,\qquad
\varkappa (\bfQ_4)=\Re \mu.
\end{equation}
Using \eqref{projq1q2} and \eqref{phi} we find that
\begin{align}\label{q1phi}
\Phi(x)\bfQ_1&=\begin{bmatrix}0&\rexp^{-\mu x}\widehat{\bfw}_2&0&\rexp^{-\mu x}\widehat{\bfw}_4\end{bmatrix},\qquad \widehat{\bfw}_2:=\begin{bmatrix}\frac12 \bfv\\ -\frac{\mu}2 \bfv\end{bmatrix},\,
\widehat{\bfw}_4:=\begin{bmatrix}-\frac1{2\mu }v\\ \frac12 v\end{bmatrix},\\
\label{q2phi}
\Phi(x)\bfQ_2&=\begin{bmatrix}\rexp^{-\nu x}\widehat{\bfw}_1&0&\rexp^{-\nu x}\widehat{\bfw}_3&0\end{bmatrix},\qquad 
\widehat{\bfw}_1:=\begin{bmatrix}\frac12 \bfu\\ -\frac{\nu}2 \bfu\end{bmatrix},\,
\widehat{\bfw}_3:=\begin{bmatrix}-\frac1{2\nu}  \bfu\\ \frac12 \bfu\end{bmatrix}.
\end{align}

In what follows, for definiteness, we assume that \eqref{remunuin} holds.  
We recall from \cite[Definition 7.2]{EJF} that the 
 generalized matrix Jost solutions are defined to be the $(4\times4)$ matrix solutions $\bfY_+^{(j)}$, $j=1,2$,  and $\bfY_-^{(j)}$, $j=3,4$, of the non-autonomous equation \eqref{ODE} that satisfy the conditions
\begin{align}\label{Jsol+}
&\limsup_{x\to+\infty}\frac1x\log\|\bfY_+^{(j)}-\Phi(x)\bfQ_j\|<\varkappa(\bfQ_j), \qquad j=1,2,\\ \label{Jsol-}
&\liminf_{x\to-\infty}\frac1x\log\|\bfY_-^{(j)}-\Phi(x)\bfQ_j\|>\varkappa(\bfQ_j), \qquad j=3,4.
\end{align}
We also define $\bfY_+(x)=\bfY_+^{(1)}(x)+\bfY_+^{(2)}(x)$ and
$\bfY_-(x)=\bfY_-^{(3)}(x)+\bfY_-^{(4)}(x)$.
Then by \cite[Definition 7.5]{EJF}, the Evans function is defined by
\begin{equation}\label{defEF}
E=E(\lambda)=\det\big(\bfY_+(0)+\bfY_-(0)\big).
\end{equation}

We concentrate first on $\bfY_+$, that is, on $j=1,2$.  By \eqref{bohls}, assertion \eqref{Jsol+} can be rephrased as follows, cf.\ \cite[Theorem 8.3(iv)]{EJF}: We want to find the solutions $\bfY_+^{(1)}$ and $\bfY_+^{(2)}$ of \eqref{ODE} such that
\begin{align}\label{Y1}
&\rexp^{\Re\mu x}\|\bfY_+^{(1)}(x)-\Phi(x)\bfQ_1\|=o(1) \text{ as $x\to+\infty$},\\\label{Y2}
&\rexp^{\Re\nu x}\|\bfY_+^{(2)}(x)-\Phi(x)\bfQ_2\|=o(1) \text{ as $x\to+\infty$}.\end{align}
Since the solutions, $\bfY_+^{(j)}$, of \eqref{ODE} decay as $x\to +\infty$,  the columns of $\bfY_+^{(j)}$ must be linear combinations of the solutions of \eqref{ODE} corresponding to $\bfp$ and $\bfq$ in \eqref{bfp} and \eqref{bfq}. Let $\bfz_k^{(j)}$, $k=1,2,3,4$,  denote the columns of the matrix $\bfY_+^{(j)}(x)$, $j=1,2$. Thus, we are looking for complex constants $a_k^{(j)}$, $b_k^{(j)}$ such that 
\begin{align}\label{Z1}
&\rexp^{\Re\mu x}\big\|
\big(a_k^{(1)}\begin{bmatrix}\bfq\\\bfq'\end{bmatrix}+b_k^{(1)}\begin{bmatrix}\bfp\\\bfp'\end{bmatrix}\big)-\bfw_k^{(1)}\big\|=o(1) \text{ as $x\to+\infty$},\\
&\rexp^{\Re\nu x}\big\|
\big(a_k^{(2)}\begin{bmatrix}\bfq\\\bfq'\end{bmatrix}+b_k^{(2)}\begin{bmatrix}\bfp\\\bfp'\end{bmatrix}\big)-\bfw_k^{(2)}\big\|=o(1) \text{ as $x\to+\infty$},
\end{align}
where $\bfp$ and $\bfq$ are the solutions from \eqref{bfp} and \eqref{bfq} and $\bfw_k^{(j)}$ are the columns of the matrix $\Phi(x)\bfQ_j$ from \eqref{q1phi} and \eqref{q2phi}.
As soon as the constants are found we will then set, cf.\ \eqref{defEF}
\begin{equation}
\bfz_k^{(j)}=a_k^{(j)}\begin{bmatrix}\bfq\\\bfq'\end{bmatrix}+b_k^{(j)}\begin{bmatrix}\bfp\\\bfp'\end{bmatrix} \text{ and } \bfY_+(x)=\begin{bmatrix} \bfz_1^{(1)}+\bfz_1^{(2)}&\bfz_2^{(1)}+\bfz_2^{(2)}&\bfz_3^{(1)}+\bfz_3^{(2)}&\bfz_4^{(1)}+\bfz_4^{(2)}\end{bmatrix}.
\end{equation}
The result of these calculations are given in the following lemma.
\begin{lemma}\label{lem:Y+} Assume \eqref{remunuin} and let
\begin{equation}\label{ab}
a_1:=\frac1{2(-\lambda-2-2\nu)},\quad
b_2:=\frac1{2(\lambda-2-2\mu)},\quad
a_3:=-\frac{a_1}{\nu},\quad
b_4:=-\frac{b_2}{\mu}.
\end{equation}
Then
\begin{gather}\label{Y+}
\bfY_+^{(1)}(x)=\begin{bmatrix} 0&b_2\begin{bmatrix}\bfp\\\bfp'\end{bmatrix}&0&b_4\begin{bmatrix}\bfp\\\bfp'\end{bmatrix}\end{bmatrix},\quad
\bfY_+^{(2)}(x)=\begin{bmatrix} a_1\begin{bmatrix}\bfq\\\bfq'\end{bmatrix}&0&a_3\begin{bmatrix}\bfq\\\bfq'\end{bmatrix}&0\end{bmatrix},\\
\bfY_+(x)=\begin{bmatrix} a_1\begin{bmatrix}\bfq\\\bfq'\end{bmatrix}&b_2\begin{bmatrix}\bfp\\\bfp'\end{bmatrix}&a_3\begin{bmatrix}\bfq\\\bfq'\end{bmatrix}&b_4\begin{bmatrix}\bfp\\\bfp'\end{bmatrix}\end{bmatrix}.\end{gather}
\end{lemma}
\begin{proof}
Formulas \eqref{bfp}, \eqref{bfq}, \eqref{bfpp}, \eqref{bfqp} yield
\begin{align}
\bfp(x)&=\rexp^{-\mu x}\big( f\cdot v+h), \quad \bfp'(x)=\rexp^{-\mu x}\big( -\mu f\cdot v+h),\\
\bfq(x)&=\rexp^{-\nu x}\big( g\cdot u+h), \quad \bfq'(x)=\rexp^{-\nu x}\big(- \nu g\cdot u+h),
\end{align}
where we define, cf.\ \eqref{defmunu},
\begin{equation}\label{fgh}
\begin{split}
f(x)&:=\lambda-2-2\mu\tanh x \to \lambda-2-2\mu =-(1+\mu)^2 \text{ as $x\to+\infty$},\\
g(x)&:=-\lambda-2-2\nu\tanh x\to-\lambda-2-2\nu=-(1+\nu)^2 \text{ as $x\to+\infty$},\end{split}\end{equation}
and denote by $h$ a generic function such that $h=h'=o(1)$ as $x\to+\infty$.

We begin with $j=1$, that is, with the solution $\bfY_+^{(1)}$. We denote by $\bfw^{(1)}_k$, $k=1,2,3,4$, the columns of the matrix $\Phi(x)\bfQ_1$ from \eqref{q1phi}. Consider the case $k=1$, that is, the first column of $\bfY_+^{(1)}$.
Recall that $w_1^{(1)}=0$ by \eqref{q1phi} and $\rexp^{(\Re\mu-\Re\nu)x}\to\infty$ as $x\to+\infty$ by \eqref{remunuin}. Thus \eqref{Z1} with $k=1$ and \eqref{fgh} imply that $a_1^{(1)}$ and $b_1^{(1)}$ must be chosen such that both $a_1^{(1)}$ and $b_1^{(1)}(1+\mu)^2$ must be equal to zero. Since $(1+\mu)^2\neq0$ by \eqref{remunuin}, we conclude that 
$a_1^{(1)}=b_1^{(1)}=0$ and just the first column of $\bfY_+^{(1)}$ is zero as required in \eqref{Y+}.

Next, we consider the case $k=2$. Now \eqref{Z1} yields $a_2^{(1)}=0$ as before. By \eqref{q1phi} the second column of $\Phi(x)\bfQ_1$ is $\rexp^{-\mu x}\widehat{\bfw}_2$; this must be compensated by a choice of $b_2^{(1)}$, that is,  one must have $b_2^{(1)}=b_2$ where $b_2$ is defined in \eqref{ab}. 

An analogous argument for the cases $k=3$ and $k=4$ shows that $a_3^{(1)}=a_4^{(1)}=b_3^{(1)}=0$ and $b_4^{(1)}=b_4$ where $b_4$ is defined in \eqref{ab}. This gives the desired result for $\bfY_+^{(1)}$. We notice that the Jost solution $\bfY_+^{(1)}$ is uniquely determined.

The argument for $a_k^{(2)}$ and $b_k^{(2)}$, $k=1,2,3,4$, and thus for $\bfY_+^{(2)}$ is analogous. Notice that for $j=2$, however, one uses \eqref{Y2}. In particular, for the second column of $\bfY_+^{(2)}$ on can choose $b_2^{(2)}$ arbitrary because 
$\rexp^{(\Re\nu-\Re\mu)x}\to0$ as $x\to\infty$ by \eqref{remunuin} (and we have chosen $b_2^{(2)}=0$). Thus, the Jost solution $Y_+^{(2)}$ is not uniquely determined, cf.\ \cite[Remark 6.3]{EJF}. Nevertheless, this does not affect the final calculation of the Evans determinant, cf.\ \cite[Lemma 7.6]{EJF}.
\end{proof}

We will now deal with $\bfY_-$, that is, with $j=3,4$. Using \eqref{Y+},
we define
\begin{equation}\label{defY-}
\bfY_-^{(3)}(x):=\widehat{\boldsymbol\sigma}_3\bfY_+^{(2)}(-x)\widehat{\boldsymbol\sigma}_3, \,\,
\bfY_-^{(4)}(x):=\widehat{\boldsymbol\sigma}_3\bfY_+^{(1)}(-x)\widehat{\boldsymbol\sigma}_3,
\text{ where $\widehat{\boldsymbol\sigma}_3:=\begin{bmatrix}\bfI_{2\times2}&0\\0&-\bfI_{2\times2}\end{bmatrix}$}.
\end{equation}
We remark that for any $(4\times 4)$ block matrix with $(2\times2)$ blocks $A,B,C,D$ one has
\begin{equation}\label{ABCD}
\widehat{\boldsymbol\sigma}_3\begin{bmatrix}A&B\\C&D\end{bmatrix}\widehat{\boldsymbol\sigma}_3
=\begin{bmatrix}A&-B\\-C&D\end{bmatrix}.\end{equation}
Noticing that $\bfR$ in \eqref{ODE} is even, and since $\bfY_+^{j-2}$ for $j=3,4$ is a solution of \eqref{ODE}, a direct computation shows that $\bfY_-^{(j)}$ just defined solves \eqref{ODE} for $j=3,4$. To deal with
\eqref{Jsol-} for $j=3$ we notice that $\widehat{\boldsymbol\sigma}_3$ is unitary, and so, replacing $x$ by $x'=-x$ and using \eqref{bohls} and $\widehat{\boldsymbol\sigma}_3\bfQ_3\widehat{\boldsymbol\sigma}_3=\bfQ_2$,
we calculate,
\begin{align}
&\liminf_{x\to-\infty}\frac1x\log\|\bfY_-^{(3)}(x)-\Phi(x)\bfQ_3\|=
\liminf_{x'\to+\infty}\frac1{-x'}\log\|\widehat{\boldsymbol\sigma}_3\bfY_-^{(2)}(x')\widehat{\boldsymbol\sigma}_3-\Phi(-x')\bfQ_3\|\\&=
-\limsup_{x\to+\infty}\frac1x\log\|\widehat{\boldsymbol\sigma}_3\big(\bfY_+^{(2)}-\Phi(x)\bfQ_2\big)\widehat{\boldsymbol\sigma}_3\|=-\limsup_{x\to+\infty}\frac1x\log\|\bfY_+^{(2)}-\Phi(x)\bfQ_2\|\\&>-\varkappa(\bfQ_2)=\varkappa(\bfQ_3),
\end{align}
as required. The argument for $j=4$ is similar. 

Finally, we compute the Evans function $E$, cf.\ \eqref{defEF}.
It follows that $\bfY(x)=\bfY_+(x)+\widehat{\boldsymbol\sigma}_3\bfY_-(-x)\widehat{\boldsymbol\sigma}_3$ and thus by
\eqref{ABCD} the matrix $\bfY(x)$ is  the block-matrix with $(2\times2)$-blocks
given by
\begin{equation}\label{YREP}
\bfY(x)=\begin{bmatrix}a_1(\bfq(x)+\bfq(-x))& b_2(\bfp(x)+\bfp(-x))&
a_3(\bfq(x)-\bfq(-x))& b_4(\bfp(x)-\bfp(-x))\\
a_1(\bfq'(x)-\bfq'(-x))& b_2(\bfp'(x)-\bfp'(-x))&
a_3(\bfq'(x)+\bfq'(-x))& b_4(\bfp'(x)+\bfp'(-x))\end{bmatrix}.
\end{equation}
In particular, $\bfY(0)$ is block diagonal and one can use \eqref{detpzero}, \eqref{detqzero} to obtain the following final result required in \eqref{eff},
\begin{align}
E&=\det \bfY(0)=\det\begin{bmatrix}2a_1\bfq(0)&2b_2\bfp(0)\end{bmatrix}\times
\det\begin{bmatrix}2a_3\bfq'(0)&2b_4\bfp'(0)\end{bmatrix}\\
&=16a_1b_2a_3b_4\det\begin{bmatrix}\bfq(0)&\bfp(0)\end{bmatrix}\times 
\det\begin{bmatrix}\bfq'(0)&\bfp'(0)\end{bmatrix}\\
&=16a_1b_2a_3b_4\mu\nu\det\begin{bmatrix}-\lambda-1&1\\1&\lambda-1\end{bmatrix}\times\det\begin{bmatrix}-\lambda+1&1\\1&\lambda+1\end{bmatrix}\\
&=16\lambda^4a_1b_2a_3b_4\mu\nu=\frac{\lambda^4
}{(1+\mu)^4(1+\nu)^4}.
\end{align}

\end{proof}

\section{General Formula for the trace of $\mathcal K(\lambda)$}\label{App:TraceK}

In  this appendix we provide a formula for the trace of the
Birman-Schwinger operator, $\mathcal K(\lambda)$, as a function of the spectral parameter, $\lambda$. This result shows that even though $\mathcal K(\lambda)$ is trace class for every $\lambda \in \mathbb C \setminus \sigma_{\text{ess}}(\mathcal L)$, the trace blows up as 
 as $\lambda$ converges to the edge of the essential spectrum.

\begin{theorem}\label{AppB:TraceKThm}
Suppose that the assumptions of Theorem~\ref{TCthm} hold.
Then 
\begin{equation}\label{AppB:TraceKGeneral}
\Tr (\mathcal K(\lambda)) \,\,=\,\,
\frac{a_+}{\sqrt{\lambda - \lambda_{e,+}}} \,\,+\,\,
\frac{a_-}{\sqrt{\lambda - \lambda_{e,-}}},
\end{equation}
where $\lambda_{e,\pm} = \delta \pm i \alpha$
are the edges of the two branches of the essential spectrum
and $a_{\pm }$ depend on the parameters in the CGLE~\eqref{cqcgle} and on $\int_\mathbb R |\boldsymbol\psi|^2$ and $\int_\mathbb R |\boldsymbol\psi|^4$,
but not on the spectral parameter, $\lambda$. 
In particular, there is a path $\lambda=\lambda(t)\to \lambda_{e,\pm}$ as $t\to\infty$ so that 
$\Re(\Tr \,\mathcal K(\lambda(t))) \to -\infty$ and hence that
$|e^{-\Tr \,\mathcal K(\lambda(t))}| \to \infty$.
In the special case of the hyperbolic secant solution of the NLSE~\eqref{yuripde}, 
\begin{equation}\label{AppB:TraceKNLSE}
    \Tr (\mathcal K(\lambda)) \,\,=\,\, -4\left[ 
\frac{1}{\sqrt{1-2i\lambda}} \,\,+\,\,
\frac{1}{\sqrt{1+2i\lambda}}\right].
\end{equation}
\end{theorem}

\begin{proof}
By \eqref{kktildej2},
  $  \mathbf{K}(x,x;\lambda) = 
   \mathbf{R}_r(x)[\mathbf{I}-\mathbf P(\lambda) \widehat{\mathbf Q}\mathbf P^{-1}(\lambda)]\mathbf{R}_{\ell}(x)$,
where $\widehat{\mathbf Q}$ is given by \eqref{qhatpform} and 
$\mathbf P(\lambda)$ and $\mathbf P^{-1}(\lambda)$
are given by \eqref{P}. Employing the blockings
\begin{equation}
    \widehat{\mathbf{Q}} = \begin{bmatrix}
    \mathbf{I} & \mathbf{0}\\ \mathbf{0} & \mathbf{0}
\end{bmatrix},\qquad
 \mathbf{P} =  \begin{bmatrix}
    \mathbf{P}_1 & \mathbf{P}_2\\ \mathbf{P}_3 & \mathbf{P}_4
\end{bmatrix},\qquad
\mathbf{P}^{-1} = \frac 14
\begin{bmatrix}
    \mathbf{S}_1 & \mathbf{S}_2\\ \mathbf{S}_3 & \mathbf{S}_4
\end{bmatrix},
\end{equation}
and using the fact that the perturbation matrices are of the form~\eqref{eq:redefRrl}, 
\begin{equation}
   \mathbf{R}_r(x) = \begin{bmatrix}
        \mathbf{T}_r(x) & \mathbf{0}\\ \mathbf{0} & \mathbf{0}
    \end{bmatrix}, \qquad\text{and}\qquad
\mathbf{R}_{\ell}(x) = \begin{bmatrix}
    \mathbf{0} & \mathbf{0}\\
\mathbf{T}_\ell(x) & \mathbf{0}
\end{bmatrix},
\end{equation}
we find that $\Tr(\mathbf{K}(x,x;\lambda))  = \frac 14 \Tr[ \mathbf P_2 \mathbf S_4 \mathbf T_\ell(x) \mathbf T_r(x)]$. 
A calculation shows that 
\begin{equation}
\mathbf P_2 \mathbf S_4 = 
\left( \frac 1{\sigma_+} + \frac 1{\sigma_-}\right)\mathbf I
-i \left( \frac 1{\sigma_+} - \frac 1{\sigma_-}\right)\mathbf J,\qquad \text{where } \mathbf J = \begin{bmatrix}0&-1\\1&0\end{bmatrix},
\end{equation}
where, by \eqref{eq:sigmapmdef}, 
 \begin{equation}
 \sigma_{\pm} \,\,=\,\, \sigma_{\pm}(\lambda)
 \,\,=\,\,
 \sqrt{\frac{1}{\det\mathbf B}  \left[\beta \mp iD/2\right] 
 \left[\lambda - (\delta \pm i \alpha)\right]}.
\end{equation}

Next, by \eqref{eq:redefRrl}, 
$\mathbf T_\ell(x) \mathbf T_r(x) = - \mathbf B^{-1}\mathbf M(x)$, where $\mathbf B^{-1} = \frac{1}{\det\mathbf B}
(\beta \mathbf I - \frac D2 \mathbf J)$ and 
$\mathbf M = \mathbf N_1 |\boldsymbol\psi|^2 + \mathbf N_2 |\boldsymbol\psi|^4 
+ (2\mathbf N_1 + 4\mathbf N_2|\boldsymbol\psi|^2)\boldsymbol\psi\boldsymbol\psi^T$,
with $\mathbf N_1 = \epsilon \mathbf I + \gamma \mathbf J$
and $\mathbf N_2 = \mu \mathbf I + \nu \mathbf J$. 
Therefore,  $-\frac 14  \mathbf P_2 \mathbf S_4 \mathbf B^{-1}
= a_1\mathbf I + a_2\mathbf J$, where 
\begin{equation}\label{eq:apm1}
a_1 = \frac {-1}{4\det\mathbf B} \left[ 
\frac{\beta - i D/2}{\sigma_+(\lambda)}  +  \frac{\beta + i D/2}{\sigma_-(\lambda)} \right]
\quad\text{and}\quad
a_2 = \frac i{4\det\mathbf B} \left[ 
\frac{\beta - i D/2}{\sigma_+(\lambda)}  - \frac{\beta + i D/2}{\sigma_-(\lambda)} \right].
\end{equation}
Since $\mathbf J^2=-\mathbf I$,
and $\mathbf J$ and $\mathbf J\boldsymbol\psi\boldsymbol\psi^T$
are traceless, we find that 
\begin{equation}
\Tr (\mathbf K(x,x;\lambda)) \,\,=\,\, 
2|\boldsymbol\psi|^2 \left[ a_1(2\epsilon + 3 \mu |\boldsymbol\psi|^2) - a_2(2\gamma+3\nu|\boldsymbol\psi|^2)\right].
\end{equation}
Finally, by \cite{GZL2025NumericalFredholm,Simon}
\begin{align}
\Tr(\mathcal K(\lambda)) \,\,&=\,\,
\int_{\mathbb R} \Tr( \mathbf K(x,x;\lambda))\, dx 
\nonumber
\\
\,\,&=\,\,
\left[ 4 \epsilon \int |\boldsymbol\psi|^2\
+ 6 \mu \int |\boldsymbol\psi|^4\right]a_1
- 
\left[ 4 \gamma \int |\boldsymbol\psi|^2
+ 6 \nu \int |\boldsymbol\psi|^4
\right]a_2.
\label{eq:apm2}
\end{align}
Equations~\eqref{AppB:TraceKGeneral} and \eqref{AppB:TraceKNLSE} now follow, 
where $a_{\pm}$ are obtained from \eqref{eq:apm1} and \eqref{eq:apm2}. The desired
path along which $\Tr(\mathcal K(\lambda(t)))$ blows up 
is given by $\lambda(t) = \lambda_{e,\pm} + 
(a_{\pm}/-t)^2$.
\end{proof}

\begin{Backmatter}


\paragraph{Funding Statement}
This work was funded by the National Science Foundation under DMS-2106203, DMS-2525548 and DMS-2106157.

\bibliography{FredholmCGLE}

@article{sinkin2003optimization,
  title={Optimization of the split-step {F}ourier method in modeling optical-fiber communications systems},
  author={Sinkin, O.V. and Holzl{\"o}hner, R. and Zweck, J. and Menyuk, C.R.},
  journal={Journal of lightwave technology},
  volume={21},
  number={1},
  pages={61},
  year={2003}
}

@article{delves1967numerical,
  title={A numerical method for locating the zeros of an analytic function},
  author={Delves, L.M. and Lyness, J.N.},
  journal={Mathematics of computation},
  volume={21},
  number={100},
  pages={543--560},
  year={1967}
}

@article{lyness1967numerical,
  title={On numerical contour integration round a closed contour},
  author={Lyness, J.N. and Delves, L.M.},
  journal={Mathematics of computation},
  volume={21},
  number={100},
  pages={561--577},
  year={1967}
}

@article{Kaup,
  title={Perturbation theory for solitons in optical fibers
         },
  author={Kaup, D.},
  journal={Phys.\ Rev.\ A },
  volume={42},
  number={9},
  pages={5689--5694},
  year={1990}
}

@book{golub2013matrix,
  title={Matrix computations},
  author={Golub, G.H. and Van {L}oan, C.F.},
  year={2013},
  publisher={John {H}opkins {U}niversity press}
}

@article{barashenkov2000oscillatory,
  title={Oscillatory instabilities of gap solitons: a numerical study},
  author={Barashenkov, IV and Zemlyanaya, EV},
  journal={Computer physics communications},
  volume={126},
  number={1-2},
  pages={22--27},
  year={2000},
  publisher={Elsevier}
}

@Article{PhysicaD124p58,
  author = "Kapitula, T. and Sandstede, B.",
  title = {Stability of bright solitary-wave solutions to perturbed nonlinear {S}chr\"odinger equations},
 journal = {Physica D}, 
 volume = {124},
 year = {1998}, 
 pages = {58--103}
}

@Article{PhysicaD116p95,
  author = "Kapitula, T.",
  title = {Stability criterion for bright solitary waves of the perturbed cubic-quintic 
           {S}chr\"odinger Equations},
 journal = {Physica D},
 volume = {116},
 year = {1998},
 pages = {95--120}
}

@article{PTRSLA340p47,
author={Pego, R.~L.  and Weinstein, M.~I.},
title={Eigenvalues, and instabilities of solitary waves},
journal={Phil. Trans. R. Soc. Lond. A},
volume={340},
pages={47--94},
year={1992}
}

@Article{SIREV48p629,
author = {Kutz, J.~N.},
title = {Mode-locked soliton lasers},
journal = {SIAM Review},
volume = {48},
number = {4},
pages = {629--678},
year={2006}
}

@Article{PRA86p033616,
author = {Shen, Y. and Kevrekidis, P.~G. and Whitaker, N. and
          Karachalios, N.~I. and Frantzeskakis, D.~J.},
title = {Finite-temperature dynamics of matter-wave dark solitons in linear and periodic potentials:
an example of an anti-damped {J}osephson junction},
journal = {Phys. Rev. A},
volume = {86},
number = {3},
pages = {033616--033628},
year={2012}
}

@Article{PRE55p4783,
  author =      "Soto-Crespo, J.~M. and Akhmediev, N.~N. and Afanasjev, V.~V. and Wabnitz, S.",
  title =       "Pulse solutions of the cubic-quintic complex {G}inzburg-{L}andau equation in the case of    
                 normal dispersion", 
  journal =     "Phys.\ Rev.\ E",
  volume =      "55",
  number =      "4",
  pages =       "4783--4796",
  year  =       "1997"
}

@Article{PLA372p3124,
  author =      "Akhmediev, N. and Soto-Crespo, J.~M. and Grelu, {Ph.}",
  title =       "Roadmap to ultra-short record high-energy pulses out of laser oscillators",
  journal =     "Phys.\ Lett.\ A",
  volume =      "372",
  pages =       "3124--3128",
  year  =       "2008"
}

@Article{JOSAB13p1439,
  author =      "Soto-Crespo, J.~M. and Akhmediev, N.~N. and Afanasjev, V.~V.",
  title =       "Stability of the pulselike solutions of the quintic complex {G}inzburg {L}andau equation",
  journal =     "J.\ Opt.\  Soc.\  Amer.\  B",
  volume =      "13",
  number =      "7",
  pages =       "1439--1449",
  year  =       "1996"
}

@Article{JOSAB31p2914,
  author =      "Wang, S. and Docherty, A. and Marks, B.~S.  and Menyuk, C.~R.",
  title =       "Boundary tracking algorithms for determining the stability of mode-locked pulses",
  journal =     "J.\ Opt.\  Soc.\  Amer.\  B",
  volume =      "31",
  number =      "11",
  pages =       "2914--2930",
  year  =       "2014"
}

@article{humpherys2006efficient,
  title={An efficient shooting algorithm for Evans function calculations in large systems},
  author={Humpherys, Jeffrey and Zumbrun, Kevin},
  journal={Physica D: Nonlinear Phenomena},
  volume={220},
  number={2},
  pages={116--126},
  year={2006},
  publisher={Elsevier}
}

@Article{JOSAB15p2757,
author = {Kapitula, T. and Sandstede, B.},
journal = {J. Opt. Soc. Am. B},
number = {11},
pages = {2757--2762},
title = {Instability mechanism for bright solitary-wave solutions to the cubic--quintic {G}inzburg--{L}andau equation},
volume = {15},
year = {1998}
}

@article{Indiana53p1095,
author={Kapitula, T. and Kutz, N. and Sandstede, B.},
title={The {Evans} function for nonlocal equations},
journal={Indiana Univ. Math. J.},
volume={53},
number={4},
pages={1095--1126},
year={2004}
}

@article{PRSL2001p257,
author={Afendikov, A.~L. and Bridges, T.~J.},
title={Instability of the {H}ocking-{S}tewartson pulse and its implications for three-dimensional
       {P}oiseulle flow},
journal={Proc. R. Soc. Lond. A},
volume={457},
pages={257--272},
year={2001}
}

@article{PhysicaD172p190,
author={Bridges, T.~J. and Derks, G. and Gottwald, G.},
title={Stability and instability of solitary waves of the fifth-order {KdV} equation:
       a numerical framework},
journal={Physica D},
volume={172},
pages={190--216},
year={2002}
}

@article{SINUM53p2329,
author={Humpherys, J. and Lytle, J.},
title={Root following in {Evans} function computation},
journal={SIAM J. Numer. Anal.},
volume={53},
pages={2329--2346},
year={2015}
}

@article{cuevas2017floquet,
  title={Floquet analysis of {K}uznetsov-{M}a breathers: {A} path towards spectral stability of rogue waves},
  author={Cuevas-Maraver, Jes{\'u}s and Kevrekidis, Panayotis G and Frantzeskakis, Dimitri J and Karachalios, Nikolaos I and Haragus, Mariana and James, Guillaume},
  journal={Physical Review E},
  volume={96},
  number={1},
  pages={012202},
  year={2017},
  publisher={APS}
}

@article{zweck2021essential,
title={The essential spectrum of periodically stationary solutions of the complex {G}inzburg-{L}andau equation},
author={Zweck, J. and  Latushkin, Y. and  Marzuola, {J.L.} and Jones, {C.R.K.T.}},
journal={J. Evol. Equ.},
year={2021},
volume={21},
pages={3313-3329}
}

@article{shinglot2023essential,
  title={The essential spectrum of periodically-stationary pulses in lumped models of short-pulse fiber lasers},
  author={Shinglot, V. and Zweck, J.},
  journal={Studies in Applied Mathematics},
  volume={150},
  number={1},
  pages={218--253},
  year={2023}}

@article{shinglot2022continuous,
 title={The Continuous Spectrum of Periodically Stationary
Pulses in a Stretched Pulse Laser},
 author={Shinglot, V. and Zweck, J. and Menyuk, C.},
journal={Opt. Lett.},
volume={47},
number={6},
pages={1490--1493},
year={2022}
}

@article{shinglot2024floquet,
  title={Floquet Stability of Periodically Stationary Pulses in a Short-Pulse Fiber Laser},
  author={Shinglot, Vrushaly and Zweck, John},
  journal={SIAM Journal on Applied Mathematics},
  volume={84},
  number={3},
  pages={961--987},
  year={2024},
  publisher={SIAM}
}

@article{latushkin2015infinite,
  title={The infinite dimensional {E}vans function},
  author={Latushkin, Y. and Pogan, A.},
  journal={Journal of Functional Analysis},
  volume={268},
  number={6},
  pages={1509--1586},
  year={2015},
  publisher={Elsevier}
}

@article{PhysicaD67p45,
author={Pego, R.~L. and Smereka, P. and Weinstein, M.~I.},
title={Oscillatory instability of travelling waves for a {KdV-Burgers} equation},
journal={Physica D},
volume={67},
pages={45--65},
year={1993}
}

@misc{GZL2025NumericalFredholm,
      title={Numerical {F}redholm determinants for matrix-valued kernels on the real line}, 
      author={Gallo, E. and Zweck, J. and Latushkin. L.},
      year={2025},
      eprint={2507.22875},
      archivePrefix={arXiv},
      primaryClass={math.NA},
      howpublished={\href{https://arxiv.org/abs/2507.22875}{https://arxiv.org/abs/2507.22875}}, 
}

@book{atkinson2008introduction,
  title={An introduction to numerical analysis},
  author={Atkinson, K.~E.},
  year={2008},
  publisher={John Wiley \& Sons}
}

@phdthesis{Gallo,
author = {Gallo, E.},
title = {Stability of solutions of the complex Ginzburg-Landau equation},
school={The University of Texas at Dallas},
year = {2024}
}

@article{zweck2024regularity,
Author = {Zweck, J. and Latushkin, L. and Gallo, E.},
Title = {A regularity condition under which integral operators with operator-valued kernels are trace class},
journal = {Bol. Soc. Mat. Mex.},
Year = {2025},
volume={31},
number={38},
pages={1--30},
note={\href{http://arxiv.org/abs//2408.04794}{http://arxiv.org/abs/2408.04794}}
}

@article{bernstein1934convergence,
  title={Sur la convergence absolue des s\'eries trigonom\'etriques},
  author={Bernstein, S.N.},
  journal={C. R. Acad. Sci. Paris},
  volume={199},
  pages={397-400},
  year={1934}
}

@incollection{akhmediev2008dissipative,
  title={Three sources and three component parts of the concept of dissipative solitons},
  booktitle={Dissipative {S}olitons: {F}rom optics to biology and medicine, Lecture Notes in Physics},
  editor={Akhmediev, N and Ankiewicz, A},
  author={Akhmediev, N and Ankiewicz, A},
  publisher={Springer, Berlin},
  volume={751},
  pages={1--28},
  year={2008}
}

@Article{RMP74p99,
author = {Aranson, I.~S. and Kramer, L.},
title = {The world of the complex {G}inzburg-{L}andau equation},
journal = {Rev. Mod. Phys.},
volume = {74},
number = {1},
pages = {99--143}, 
year={2002}
}

@article{shen2016spectra,
  title={Spectra of Short Pulse Solutions of the Cubic-Quintic Complex {G}inzburg-{L}andau
         Equation near Zero Dispersion},
  author={Shen, Y. and Zweck, J. and Wang, S. and Menyuk, C.R.},
  journal={Stud. Appl. Math.}, 
  volume={137},
  number={2},
  pages={238--255},
  year={2016},
  publisher={Wiley Online Library}
}

@book{Simon,
author={B. Simon},
title={Trace Ideals and Their Applications},
edition={2nd},
year={2005},
publisher={American Mathematical Society},
address={Providence, RI}
}

@book{TeschlFA,
author={G. Teschl},
title={Topics in Real and Functional Analysis},
year={2010},
publisher={American Mathematical Society}, address={Providence, RI}
}

@book{Kap,
 author = {Kapitula, T. and Promislow, K.},
 title={Spectral and Dynamical Stability of Nonlinear Waves},
 year={2013},
 publisher={Springer, New York, NY}
}

@article{EssSpec,
title={The essential spectrum of periodically stationary solutions of the complex {G}inzburg-{L}andau equation},
author={Zweck, J. and  Latushkin, Y. and  Marzuola, {J.L.} and Jones, {C.R.K.T.}},
journal={J. Evol. Equ.},
year={2021},
volume={21},
pages={3313-3329}
}

@article{EJF,
  title={Evans functions, {J}ost functions, and {F}redholm determinants},
  author={Gesztesy, F. and Latushkin, Y. and Makarov, K.A.},
  journal={Archive for Rational Mechanics and Analysis},
  volume={186},
  number={3},
  pages={361--421},
  year={2007},
  publisher={Springer}
}

@article{TURITSYN2012,
title = {Dispersion-managed solitons in fibre systems and lasers},
journal = {Physics Reports},
volume = {521},
number = {4},
pages = {135-203},
year = {2012},
note = {Dispersion-Managed Solitons in Fibre Systems and Lasers},
issn = {0370-1573},
doi = {https://doi.org/10.1016/j.physrep.2012.09.004},
url = {https://www.sciencedirect.com/science/article/pii/S0370157312002657},
author = {S. K. Turitsyn and B. G. Bale and M. P. Fedoruk}
}

@article{Palmer,
author={Palmer, K. J.}, 
title={Exponential dichotomies and {F}redholm operators}, 
journal={Proc. Amer. Math. Soc.},
volume={104},
number={1},
pages={149-156},
year={1988}
}

@article{Bornemann,
  	year = 2009,
  	publisher = {American Mathematical Society ({AMS})},
  	volume = {79},
  	number = {270},
  	pages = {871--915},
  	author = {Bornemann, F.},
  	title = {On the numerical evaluation of {F}redholm determinants},
  	journal = {Mathematics of Computation}
}

@book{Meyer,
author={Meyer, C.},
title={Matrix Analysis and Applied Linear Algebra},
publisher={SIAM},
year={2001},
address={Philadelphia, PA}}

@book{Kato,
author={Kato, T.},
title={Perturbation Theory for Linear Operators},
year={1980},
publisher={Springer-Verlag},
address={Berlin}
}

@book{GGK,
author={Gohberg, I. and Goldberg, S. and Krupnik, N.},
title={Traces and Determinants of Linear Operators},
year={2000},
edition={1st},
publisher={Birkh\"auser Basel},
address={Basel, Switzerland}
}

@book{EdmundsEvans,
    author = {Edmunds, D.~E. and Evans, W.~D.},
    title = "{Spectral Theory and Differential Operators}",
    publisher = {Oxford University Press},
    year = {2018},
    month = {05},
    isbn = {9780198812050},
    doi = {10.1093/oso/9780198812050.001.0001},
    url = {https://doi.org/10.1093/oso/9780198812050.001.0001},
}

@article{2,
  title={Ultrafast fiber lasers based on self-similar pulse evolution: a review of current progress},
  author={Chong, Andy and Wright, Logan G and Wise, Frank W},
  journal={\rpp},
  volume={78},
  number={11},
  pages={113901},
  year={2015},
  publisher={IOP Publishing}
}

\end{Backmatter}

\end{document}